\documentclass[reqno,11pt]{amsart}

\usepackage[utf8]{inputenc}

\usepackage[all]{xy}
\usepackage{tikz-cd}
\usepackage{amssymb}
\usepackage{amsgen}
\usepackage{amsmath}
\usepackage{amsthm}
\usepackage{mathrsfs}
\usepackage{cite}
\usepackage{amsfonts}

\usepackage[
margin=1.1in,
marginpar=2cm,
includefoot,
footskip=30pt,
]{geometry}

\newcommand{\skel}[1]{^{(#1)}}

\newcommand{\dom}{\mathop{\boldsymbol d}}
\newcommand{\ran}{\mathop{\boldsymbol r}}

\newcommand{\supp}{\mathop{\mathrm{supp}}}

\newcommand{\Ind}{\mathop{\mathrm{Ind}}\nolimits}

\newcommand{\Lc}{\ensuremath {\mathcal{L}_c(X)} }
\newcommand{\OO}{\mathcal O}

\newcommand{\inv}{^{-1}}

\newcommand{\til}[1]{\ensuremath{\widetilde {#1}}}

\newcommand{\wh}{\widehat}

\newcommand{\module}[1]{#1\text{-}\mathbf{mod}}

\usepackage{xcolor}

\def\LLL{\mathcal L} 

\newcommand{\grp}{\mathscr G}
\newcommand{\Ann}{\mathop{\mathrm{Ann}}\nolimits}
\newcommand{\secmod}{\Gamma_c(\mathscr G,\mathcal M)}
\newcommand{\secring}{\Gamma_c(\mathscr G,\mathcal O)}
\newcommand{\secringunit}{\Gamma_c(\mathscr G^{(0)},\mathcal O)}
\newcommand{\scj}{\subseteq}

\newcommand{\cn}{\mathbb{C}}

\newtheorem{Thm}{Theorem}[section]
\newtheorem{Prop}[Thm]{Proposition}

\newtheorem{Lemma}[Thm]{Lemma}
{\theoremstyle{definition}
\newtheorem{Def}[Thm]{Definition}}
{\theoremstyle{remark}
\newtheorem{Rmk}[Thm]{Remark}}
\newtheorem{Cor}[Thm]{Corollary}
{\theoremstyle{remark}
}
{\theoremstyle{remark}
\newtheorem{Example}[Thm]{Example}}

{\theoremstyle{remark}
}
{\theoremstyle{remark}
}
{\theoremstyle{remark}
}

\numberwithin{equation}{section}

\makeatletter
\@namedef{subjclassname@2020}{%
  \textup{2020} Mathematics Subject Classification}
\makeatother

\title[Ideals of \'etale groupoid algebras with coefficients in a sheaf]{Ideals of \'etale groupoid algebras with coefficients in a sheaf with applications to topological dynamics}

\author{Gilles G. de Castro, Daniel Gon\c{c}alves\and Benjamin Steinberg}

\address[G. G. de Castro]{%
	Departmento de Matem\'atica\\
	Universidade Federal de Santa Catarina\\
	Florian\'{o}polis, SC, 88040-900\\
	Brazil}
	\email{gilles.castro@ufsc.br}

\address[D.~Gon\c{c}alves]{%
	Departmento de Matem\'atica\\
	Universidade Federal de Santa Catarina\\
	Florian\'{o}polis, SC, 88040-900\\
	Brazil}
	\email{daemig@gmail.com}	

\address[B.~Steinberg]{%
    Department of Mathematics\\
    City College of New York\\
    Convent Avenue at 138th Street\\
    New York, New York 10031\\
    USA}
\email{bsteinberg@ccny.cuny.edu}

\thanks{The third author thanks the Fulbright Commission for its support for visiting the Federal University of Santa Catarina in Brazil as well as PSC CUNY. The first and second authors were partially supported by Capes-PrInt-Brazil. The second author was also supported by CNPq- Brazil and  Funda\c{c}\~ao de Amparo \`a Pesquisa e Inova\c{c}\~ao do Estado de Santa Catarina (FAPESC)}
\date{\today}


\begin{document}

\begin{abstract}
We prove the Effros-Hahn conjecture for groupoid algebras with coefficients in a sheaf, obtaining as a consequence a description of the ideals in skew inverse semigroup rings. We also use the description of the ideals to characterize when the groupoid algebras with coefficients in a sheaf are von Neumann regular, primitive, semiprimitive, or simple. We apply our results to the topological dynamics of actions of inverse semigroups, describing the existence of dense orbits and minimality in terms of primitivity and simplicity, respectively, of the associated algebra. Moreover, we apply our results to the usual complex groupoid algebra of continuous functions with compact support, used to build the C*-algebra associated with a groupoid, and describe criteria for its simplicity.  
\end{abstract}

\maketitle

\section{Introduction}

Groupoid algebras with coefficients in a sheaf unify the study of usual convolution groupoid algebras (often called Steinberg algebras) and skew inverse semigroup rings (see \cite{BenDan}). The aforementioned constructions are key in the study of algebras associated with combinatorial objects such as graphs, higher rank graphs, ultragraphs, etc, and have deep connections with topological dynamics and the intrinsic dynamics associated with combinatorial objects, see \cite{BC, cp} for example. Among the relevant properties of convolution groupoid algebras and inverse semigroup skew rings, the constitution of its ideals play a crucial role. For just a couple of examples, the ideal structure of a Leavitt path algebra may be recovered from the ideal structure of a convolution groupoid algebra, see \cite{HRove}, and minimality and topological freeness of actions of inverse semigroups can be described in terms of simplicity of the associated skew rings (see \cite{BGOR, DG, GOR}). Our goal in this paper is to obtain a description of the ideals in groupoid algebras with coefficients in a sheaf and apply this description to the topological dynamics of actions of inverse semigroups and to the usual complex groupoid algebra of continuous functions with compact support used to build the C*-algebra associated with a groupoid. 

When one searches the literature for a classification of ideals in crossed products, the references go back to the original Effros-Hahn conjecture, which suggested that every primitive ideal of a crossed product of an amenable locally compact group with a commutative C*-algebra should be induced from a primitive ideal of an isotropy group, see \cite{Effros}. Since then, the conjecture has been proved and generalized in several contexts, see \cite{Ben} for a comprehensive account of developments. 

In the purely algebraic setting, an Effros-Hahn type conjecture is proved for partial skew group rings in \cite{Doke} and for groupoid convolution algebras in \cite{Paulinho} and \cite{Ben}. In this paper, we prove an algebraic Effros-Hahn type conjecture in the context of groupoid algebras with coefficients in a sheaf. Therefore, we extend the known results of \cite{Paulinho, Doke, Ben}  to include skew inverse semigroup rings, and at the same time provide a unified statement to the Effros-Hahn conjecture proved in \cite{Paulinho, Doke, Ben}. 

After we show our version of the Effros-Hanh conjecture, we use the machinery developed to prove it, namely the induction of modules for groupoid algebras with coefficients in a sheaf,  to describe several algebraic properties of groupoid algebras with coefficients in a sheaf. This includes simplicity,  semiprimitivity, and primitivity. For topological actions of inverse semigroups, since the associated skew inverse semigroup ring can be seen as a groupoid convolution algebra with coefficients in a sheaf, we relate topological properties of the action with algebraic properties of the associated groupoid convolution algebra. Furthermore, we realize the usual algebra $C_c(\mathscr G)$  (used to build groupoid $C^*$-algebras) as a groupoid algebra with coefficients in a sheaf (and hence as a skew inverse semigroup ring), and then apply our topological results to describe when it is simple in terms of the groupoid. Therefore, we provide a bridge between algebra and analysis (for example, our simplicity characterization of $C_c(\mathscr G)$ should be compared with the characterization of simplicity of the reduced groupoid C*-algebra).

We now give a more detailed description of our work.

We begin by presenting some necessary background on the topics of the paper. In particular, we recall the construction of convolution algebras with coefficients in a sheaf and of skew inverse semigroup rings. Furthermore, we recall the Disintegration theorem, which is used in \cite{BenDan} to prove that skew inverse semigroup rings and groupoid convolution algebras with coefficients in a sheaf are essentially the same objects. 

We describe the induction process that transforms representations of the skew group rings associated with isotropy groups to modules for groupoid algebras with coefficients in a sheaf in Section~\ref{induced}, and use this to prove Theorem~\ref{thm:ann.ideal}, which asserts that every ideal in a groupoid algebra with coefficients in a sheaf of modules is an intersection of annihilators of induced modules.

In section~\ref{s:applications}, we prove the Effros-Hahn conjecture for 
groupoid convolution algebras with coefficients in a sheaf, Theorem~\ref{t:left.max},
and use Theorem~\ref{thm:ann.ideal} to prove a number of properties of $\secring$, the groupoid convolution algebra associated with a sheaf. We start by showing that every primitive ideal of $\secring$ is the annihilator of a single induced representation, Theorem~\ref{thm:primitive.ideal}. We then show that induction of modules preserves simplicity, Theorem~\ref{thm:simple.module}. As with usual groupoid algebras, there is a diagonal commutative algebra inside $\secring$, call it $\secringunit$. We characterize when $\secringunit$ is von Neumann regular in Proposition~\ref{p:vnr}, and describe when it is maximal commutative inside $\secring$ in Proposition~\ref{p:neweffective}. We specify the role of $\secringunit$ in determining when a ring homomorphism from $\secring$ is injective in the Generalized Uniqueness Theorem, Theorem~\ref{gut}. We describe primitivity of $\secring$ in Theorem~\ref{t:main.primitive} and, under some assumptions, show that $\secring$ is left primitive if, and only if, the unit space of the groupoid has a dense orbit, Theorem~\ref{t:max.abel.case}. In Theorems~\ref{t:main.semiprimitive} and \ref{tsemiprimitivity} we give a sufficient condition for $\secringunit$ to be semiprimitive and in Theorem~\ref{simplelife} we characterize simplicity of $\secringunit$.

We devote Section~\ref{dynamics} to the topological dynamics of actions of inverse semigroups. In Proposition~\ref{oneway}, we show that if the associated algebra $\Gamma_c(S\ltimes X,\OO)$ is left primitive then the action has a dense orbit. From this, we obtain that for certain topologically free actions on locally compact, Hausdorff, zero-dimensional spaces the associated algebra $\Gamma_c(S\ltimes X,\OO)$ is left primitive if, and only if, the action has a dense orbit, Corollary~\ref{primodenso}. In Propositions~\ref{semiprop1} and \ref{semiprop2} we give sufficient conditions for semiprimitivity of $\Gamma_c(S\ltimes X,\OO)$, and, in Theorem~\ref{simpleaction} we relate simplicity of $\Gamma_c(S\ltimes X,\OO)$ with minimality of the action. 

Finally, in Section~\ref{complex}, we study the usual algebra $C_c(\grp)$ of complex valued, continuous functions with compact support, which is used to build the C*-algebras (full and reduced) associated to a groupoid. For Hausdorff groupoids, we give a direct description of $C_c(\grp)$ as a skew inverse semigroup ring, Proposition~\ref{ccgpd}. For general groupoids, we provide a realization of $C_c(\grp)$ as a groupoid convolution algebra with coefficient in a sheaf in Theorem~\ref{thm:groupoid.ring.sheaf}; this implies that the same skew inverse semigroup ring representation as in the Hausdorff case holds in general. Applying results of the previous section to $C_c(\grp)$, in the Hausdorff case, we characterize the simplicity of $C_c(\grp)$ in terms of minimality and effectiveness of $\grp$, see Theorem~\ref{thm:C_c(G) simple}. We finish the paper considering the groupoid arising from a partial action and prove that $C_c(\grp)$ is a partial crossed product, where $\grp$ is the transformation groupoid of the partial action.


\section{Background}

For completeness, in this section we recall the relevant concepts that will be used throughout the paper, as defined in \cite{BenDan}.

\subsection{Groupoids}

A \emph{groupoid} $\mathscr G$ is a small category of isomorphisms. A \emph{topological groupoid} is a groupoid equipped with a topology making the multiplication and inverse operations continuous. The elements of the form $gg^{-1}$ are called \emph{units}. We denote the set of units of  $\mathscr G$ by $\mathscr G^{(0)}$, and refer to $\mathscr G^{(0)}$ as the \emph{unit space}.  The \emph{source} and \emph{range} maps are given by $\dom(g)=g^{-1}g$ and $\ran(g)= gg^{-1}$, for $g\in G$. These maps are necessarily continuous when $G$ is a topological groupoid.

An \emph{\'etale groupoid} is a topological groupoid  $\mathscr G$ such that its unit space  $\mathscr G\skel 0$ is locally compact and Hausdorff and its range map $\ran$ is a local homeomorphism (this implies that the domain map $\dom$ and the multiplication map are also local homeomorphisms). A \emph{bisection} of $\mathscr G$ is a subset $B\subseteq \mathscr G$ such that the restriction of the range and source maps to $B$ are injective. An \'etale groupoid is \emph{ample} if its unit space has a basis of compact open sets or, equivalently, if the arrow space $\mathscr G\skel 1$ has a basis of compact open bisections (when it is clear from the context, we will also use $\grp$ to denote the arrow space). We remark for future use that, for every open bisection $B$, of an ample groupoid $\mathscr G$, the range and source maps are homeomorphisms from $B$ to $\dom(B)$ and $\ran(B)$, respectively. The \emph{isotropy group} of a unit $x\in \mathscr G^{(0)}$ is the group $\mathscr G_x^x = \{g\in \mathscr G \mid \dom(g)=\ran(g)=x\}$. The \emph{isotropy subgroupoid} of a groupoid $\mathscr G$ is
the subgroupoid \[\mathrm{Iso}(\mathscr G)=\bigcup_{x\in \mathscr G^{(0)}} \mathscr G_x^x =\{g\in \mathscr G: \dom(g)=\ran(g)\}.\] We say that the topological groupoid $\mathscr G$ is \emph{effective} if $ \mathscr G^{(0)}= \mathrm{int}(\mathrm{Iso}(\mathscr G))$.
 We will use the notation $\zeta:y\to z$ meaning that $\dom(\zeta)=y$ and $\ran(\zeta)=z$. Furthermore, for each $x \in \mathscr G\skel 0$, $\mathrm{Orb}(x) = \ran\circ \dom^{-1}(x)$.

From now on, following Bourbaki, the term ``compact'' will include the Hausdorff axiom.  However, a space can be locally compact without being Hausdorff.
If $f\colon X\to Z$ and $g\colon Y\to Z$ are maps of spaces, then their \emph{pullback} is \[X\times_{f,g} Y=\{(x,y)\mid f(x)=g(y)\}\] (with the subspace topology of the product space).

\subsection{Ample groupoid convolution algebras with coefficients in a sheaf of rings}\label{s:sheaf.coeff}
In this section, we recall the convolution algebra associated with a sheaf of rings over an ample groupoid\footnote{Technically, this is just a ring but convolution algebra fits better with terminology in operator algebras and every ring is a $\mathbb Z$-algebra}. 

Let $\mathscr G$ be an ample groupoid.  Then a \emph{$\mathscr G$-sheaf} $\mathcal E$ consists of a topological space $E$, a local homeomorphism $p\colon E\to \mathscr G\skel 0$ and a continuous map $\alpha\colon \mathscr G\skel 1\times_{\dom,p} E\to E$ (written $(\gamma,e)\mapsto \alpha_{\gamma}(e)$) satisfying the following axioms:
\begin{itemize}
\item [(S1)] $\alpha_{p(e)}(e) = e$;
\item [(S2)] $p(\alpha_{\gamma}(e))=\ran(\gamma)$ if $\dom(\gamma)=p(e)$;
\item [(S3)] $\alpha_{\beta}(\alpha_{\gamma}(e)) = \alpha_{\beta\gamma}(e)$ whenever $\dom(\beta)=\ran(\gamma)$ and $\dom(\gamma)=p(e)$.
\end{itemize}								
If $x\in \mathscr G\skel 0$, then $\mathcal E_x=p^{-1}(x)$ is called the \emph{stalk} of $\mathcal E$ at $x$.  Notice that $\alpha_{\gamma}\colon \mathcal E_{\dom(\gamma)}\to \mathcal E_{\ran(\gamma)}$ is a bijection with inverse $\alpha_{\gamma\inv}$.   


We shall be interested in sheaves with extra structure.  A \emph{$\mathscr G$-sheaf of (unital) rings} is a $\mathscr G$-sheaf $\mathcal O=(E,p,\alpha)$ equipped with a unital ring structure on each stalk $\mathcal O_x$ such that the following axioms hold:
\begin{itemize}
\item [(SR1)] $+\colon E\times_{p,p} E\to E$ is continuous;
\item [(SR2)] $\cdot\colon E\times_{p,p} E\to E$ is continuous;
\item [(SR3)] the unit section $x\mapsto 1_x$ is a continuous mapping $\mathscr G\skel 0\to E$;
\item [(SR4)] $\alpha_{\gamma}\colon \mathcal O_{\dom(\gamma)}\to \mathcal O_{\ran(\gamma)}$ is a ring homomorphism for all $\gamma\in \mathscr G\skel 1$.
\end{itemize}

Note that the zero section $x\mapsto 0_x$ is continuous and that the negation map is continuous (these are standard facts about sheaves of abelian groups, and hence rings, over spaces,~cf.~\cite{Dowker}).

Given a $\mathscr G$-sheaf of rings $\mathcal O=(E,p,\alpha,+,\cdot)$, next we recall the definition of the \emph{ring of global sections of $\mathcal O$ with compact support}, which we shall also call the \emph{convolution algebra of $\mathscr G$ with coefficients in the sheaf of rings $\mathcal O$}.
Let $A(\mathscr G,\mathcal O)$ be the set of all mappings $f\colon \mathscr G\skel 1\to E$ such that $p\circ f=\ran$, that is, $f(\gamma)\in \mathcal O_{\ran(\gamma)}$ for all $\gamma\in \mathscr G\skel 1$. Equip $A(\mathscr G,\mathcal O)$ with a binary operation by defining $(f+g)(\gamma)= f(\gamma)+g(\gamma)$, which we refer to as pointwise addition. With this operation, $A(\mathscr G,\mathcal O)$ is an abelian group with respect to pointwise addition with $0$ as the identity and $(-f)(\gamma) = -f(\gamma)$ for $\gamma\in \mathscr G\skel 1$.

We define, as an abelian group, $\Gamma_c(\mathscr G,\mathcal O)$ to be the subgroup generated by all mappings $f\in A(\mathscr G,\mathcal O)$ such that there is a compact open bisection $U$ with $f|_U$ continuous and $f|_{\mathscr G\skel 1\setminus U} = 0$. In this case, we say that $f$ is \emph{supported} on $U$. If $U$ is a compact open bisection and $s\colon \ran(U)\to E$ is any (continuous) section of $p$, then we can define an element $s\chi_U\in \Gamma_c(\mathscr G,\mathcal O)$, supported on $U$, by
\[(s\chi_U)(\gamma) = \begin{cases} s(\ran(\gamma)), & \text{if}\ \gamma\in U\\ 0_{\ran(\gamma)}, & \text{else.}\end{cases} \]  In the special case that $s$ is the unit section $x\mapsto 1_x$ over $U$, we denote $s\chi_U$ by simply $\chi_U$.  In other words,
\[\chi_U(\gamma) = \begin{cases} 1_{\ran(\gamma)}, & \text{if}\ \gamma\in U\\ 0_{\ran(\gamma)}, & \text{else.}\end{cases}\]

Notice that if $f\in \Gamma_c(\mathscr G,\mathcal O)$ is supported on a compact open bisection $U$, then $f=s\chi_U$ where $s=f\circ (\ran|_U)\inv$.  Thus $\Gamma_c(\mathscr G,\mathcal O)$ can also be described as the abelian group generated by all elements of the form $s\chi_U$ where $s\colon \ran(U)\to E$ is a section, and $U$ is a compact open bisection.

A crucial property of elements of $\Gamma_c(\mathscr G,\mathcal O)$, and that we will use in our work, is that they can only be non-zero on finitely many points of any fiber of $\dom$ or $\ran$.

\begin{Prop}\label{p:finiteness.prop}
Let $f\in \Gamma_c(\mathscr G,\mathcal O)$ and $x\in \mathscr G\skel 0$.  Then there are only finitely many $\gamma\in \dom\inv (x)$ such that $f(\gamma)\neq 0$ and, similarly, for $\ran\inv(x)$.
\end{Prop}

Finally, to make $\Gamma_c(\mathscr G,\mathcal O)$ into a ring, we define the convolution of elements of $\Gamma_c(\mathscr G,\mathcal O)$ as follows. 
If $f,g\in \Gamma_c(\mathscr G,\mathcal O)$ and $\gamma\in \mathscr G\skel 1$, then
\begin{equation}\label{eq:define.conv}
f\ast g(\gamma) = \sum_{\beta\rho=\gamma}f(\beta) \alpha_{\beta}(g(\rho)).
\end{equation}

It is proved in \cite{BenDan} that 
if $f,g\in \Gamma_c(\mathscr G,\mathcal O)$, then $f\ast g\in \Gamma_c(\mathscr G,\mathcal O)$. In fact it is shown that if $f$ is supported on $U$ and $g$ is supported on $V$, with $U,V$ compact open bisections, then $f\ast g$ is supported on $UV$.  In the case that $\grp$ is Hausdorff, it is shown in~\cite{BenDan} that $\secring$  consists precisely of those continuous functions $f\colon \grp^{(1)}\to E$ with $p\circ f=\ran$ and compact support (i.e., the inverse image of the complement of the zero section is compact).   Furthermore, in \cite{BenDan} it is shown how to build a sheaf of rings so that the above construction yields the usual algebra of $\mathscr G$ over a unital ring $R$ from~\cite{Steinbergalgebra}. This sheaf is recalled below.

\begin{Example}\label{ex:constant.sheaf}
 Let $R$ be any unital ring, which we view as a space with the discrete topology.  We define the \emph{constant sheaf} of rings $\Delta(R)$ to be the $\mathscr G$-sheaf of rings with $E=R\times \mathscr G\skel 0$ and with $p\colon R\times \mathscr G\skel 0\to\mathscr G\skel 0$ the projection.  The addition and multiplication are pointwise, that is, $(r,x)+(r',x) = (r+r',x)$ and $(r,x)(r',x) = (rr',x)$.  The mapping $\alpha$ is given by $\alpha(\gamma)(r,\dom(\gamma)) = (r,\ran(\gamma))$.  Then $\Gamma_c(\mathscr G,\Delta(R))$ is the usual algebra of $\mathscr G$ over $R$, from~\cite{Steinbergalgebra}.
\end{Example}

\subsection{The disintegration theorem}\label{desintegration}

A key result in \cite{BenDan} is the disintegration theorem, which generalizes results in \cite{groupoidbundles} for Steinberg algebras. In this subsection, we recall the theorem as well as the related concepts.

Let $\mathscr G$ be an ample groupoid and $\mathcal O=(E,p,\alpha)$ be a $\mathscr G$-sheaf of rings.  Put $R=\Gamma_c(\mathscr G,\mathcal O)$; it is a ring with local units.  A (left) $R$-module $M$ is \emph{unitary} if $RM=M$.  We denote by $\module{R}$ the category of unitary (left) $R$-modules. 

A \emph{$\mathscr G$-sheaf of $\mathcal O$-modules} $\mathcal M=(F,q,\beta)$ is a $\mathscr G$-sheaf such that each stalk $\mathcal M_x$ has a (unitary) left $\mathcal O_x$-module structure such that:
\begin{enumerate}
  \item [(SM1)] addition $+\colon F\times_{q,q} F\to F$ is continuous;
  \item [(SM2)] the  module action $E\times_{p,q} F\to F$ is continuous;
  \item [(SM3)] $\beta_{\gamma}(rm) = \alpha_{\gamma}(r)\beta_{\gamma}(m)$ for all $r\in \mathcal O_{\dom(\gamma)}$ and $m\in \mathcal M_{\dom(\gamma)}$.
\end{enumerate}

The following result is proved in \cite{BenDan} and will be used in our text.

\begin{Prop}\label{p:lots.of.secs}
Let $X$ be a Hausdorff space with a basis of compact open sets and let $\mathcal A=(F,q)$ be a sheaf of abelian groups on $X$. Then, for each $f\in \mathcal A_x$, there is a section $s\colon X\to F$ with compact support such that $s(x)=f$.
\end{Prop}

If $\mathcal M$ is a $\mathscr G$-sheaf of $\mathcal O$-modules, then we can look at the set $M=\Gamma_c(\mathscr G,\mathcal M)$ of continuous sections $s\colon \mathscr G\skel 0\to F$ with compact support. This is an abelian group with pointwise operations.  We define an $R$-module structure on it by putting, for $f\in R$ and $m\in M$,
\[(fm)(x) = \sum_{\gamma\in \ran\inv(x)}f(\gamma)\beta_{\gamma}(m(\dom(\gamma))).\]

\begin{Prop}\label{p:is.unitary}
The construction $\mathcal M\longmapsto \Gamma_c(\mathscr G,\mathcal M)$ is a functor from the category of $\mathscr G$-sheaves of $\mathcal O$-modules to the category of $\module{\Gamma_c(\mathscr G,\mathcal O)}$.
\end{Prop}

A quasi-inverse for the functor above is constructed in \cite{BenDan}. We recall the main concepts below.

Let $M$ be a unitary $R$-module and $x\in \mathscr G\skel 0$. Let \[N_x = \{m\in M\mid \chi_Um=0\ \text{for some}\ U\subseteq \mathscr G\skel 0\ \text{compact open with}\ x\in U\},\] and define
$ M_x:= M/N_x$. Furthermore, define \[F= \coprod_{x\in \mathscr G\skel 0} M_x\] and $q\colon F\to \mathscr G\skel 0$ by $q([m]_x) = x$.  Put a topology on $F$ by taking as a basis all sets of the form
\[D(m,U) = \{[m]_x\mid x\in U\},\] where $m\in M$ and $U\subseteq \mathscr G\skel 0$ is compact open.  To define the $\mathscr G$-sheaf structure, for $\gamma\in \mathscr G$, put \[\beta_{\gamma}([m]_{\dom(\gamma)}) = [\chi_Um]_{\ran(\gamma)},\] where $U$ is any compact open bisection containing $\gamma$.  Then $\beta\colon \mathscr G\skel 1\times_{\dom,q} F\to F$ is well defined, continuous, and turns $$\mathrm{Sh}(M) := (F,q,\beta)$$ into a $\mathscr G$-sheaf.  
Moreover, $\mathrm{Sh}(M)$ is a  $\mathscr G$-sheaf of abelian groups with respect to the fiberwise addition $[m]_x+[n]_x=[m+n]_x$ (i.e., addition is fiberwise continuous and each $\beta_\gamma$ is an additive homomorphism). To define the $\mathcal O_x$-module structure on $M_x$, let $r\in \mathcal O_x$. Choose a section $t\in \Gamma_c(\mathscr G\skel 0,\mathcal O)$ with $t(x) = r$ (using Proposition~\ref{p:lots.of.secs}), and define $r[m]_x = [tm]_x$. Then $\mathrm{Sh}(M)$ is a $\mathscr G$-sheaf of $\mathcal O$-modules and, as proved in \cite{BenDan}, we have that the construction  $M\longmapsto \mathrm{Sh}(M)$ is a functor, as we make precise below.

\begin{Prop}\label{p:mod.to.sheaf}
The construction $M\longmapsto \mathrm{Sh}(M)$ is a functor from the category $\module{\Gamma_c(\mathscr G,\mathcal O)}$ to the category of $\mathscr G$-sheaves of $\mathcal O$-modules.
\end{Prop}

Combining the last two propositions we get.

\begin{Thm}[Disintegration theorem]\label{t:disint}
The functors $M\longmapsto \mathrm{Sh}(M)$ and $\mathcal M\longmapsto \Gamma_c(\mathscr G,\mathcal O)$ provide an equivalence between the category $\module{\Gamma_c(\mathscr G,\mathcal O)}$ of unitary $\Gamma_c(\mathscr G,\mathcal O)$-modules and the category of $\mathscr G$-sheaves of $\mathcal O$-modules.
\end{Thm}

It follows from the above that there are natural isomorphisms $M\cong \Gamma_c(\mathscr G,\mathrm{Sh}(M))$ and $\mathcal M\cong \mathrm{Sh}(\Gamma_c(\mathscr G,\mathcal M))$.

\begin{Rmk}\label{r:find.ideals}
As a consequence of Theorem~\ref{t:disint}, $M\cong \Gamma_c(\mathscr G,\mathrm{Sh}(M))$ for any unitary module $M$ and so we have an equality of annihilator ideals \[\Ann(M)=\Ann(\Gamma_c(\mathscr G,\mathrm{Sh}(M))).\] In particular, for every ideal $I$ of $ \secring$ we have \[I=\Ann(\secring /I)=\Ann(\Gamma_c(\mathrm{Sh}(\secring/I))).\] 
\end{Rmk}

\subsection{Skew inverse semigroup rings}

By a \emph{partial automorphism} of a ring $A$, we mean a ring isomorphism $\varphi\colon I\to J$ between two-sided ideals $I,J$ of $A$.  The collection of all partial automorphisms of $A$ forms an inverse monoid that we denote $I_A$.  If $S$ is an inverse semigroup, then an \emph{action} of $S$ on $A$ is a homomorphism $\alpha\colon S\to I_A$, usually written $s\mapsto \alpha_s$.  The domain of $\alpha(s)$ is denoted $D_{s^*}$ and the range is then $D_s$.  We say that the action is \emph{non-degenerate} if \[\sum_{e\in E(S)} D_e=A.\]

 To ensure associativity of the skew inverse semigroup ring, we assume that each $D_s$ is a ring with local units (although weaker conditions suffice). Given an action $\alpha$ of $S$ on a ring $A$,
the construction of the corresponding skew inverse semigroup ring is done in three steps.
\begin{enumerate}
	\item First we consider the set
\begin{equation}\label{eq:define.L}
	\mathcal{L} = \left\{ \sum_{s\in S}^{\text{finite}} a_s \delta_s \mid \ a_s \in D_s \right\}\cong \bigoplus_{s\in S} D_s
\end{equation}
where $\delta_s$, for $s\in S$, is a formal symbol (and $0\delta_s=0$).
We equip $\mathcal{L}$ with component-wise addition and with multiplication
defined as the linear extension of the rule
\[
	(a_s \delta_s)(b_t \delta_t) = \alpha_{s}(\alpha_{s^*}(a_s) b_t) \delta_{st}.
\]

	\item Then, we consider the ideal
	\begin{equation}\label{eq:define.N}
	\mathcal{N} = \langle a \delta_r - a \delta_s \mid r,s \in S, \ r \leq s \text{ and } a \in D_r \rangle,
\end{equation}
i.e., $\mathcal{N}$ is the ideal of $\mathcal{L}$ generated by all elements of the form $a\delta_r - a\delta_s$, where $r\leq s$ and $a \in D_r$.   It is shown in~\cite[Lemma~2.3]{BGOR}, that these elements already generate $\mathcal{N}$ as an additive group.
	\item Finally, we define the corresponding
	\emph{skew inverse semigroup ring}, which we denote by $A\rtimes S$, as the quotient ring $\mathcal{L}/\mathcal{N}$.
\end{enumerate}
If $S$ is a group, then the ideal $\mathcal N$ is the zero ideal and the multiplication simplifies to the rule $a\delta_s\cdot b\delta_t = a\alpha_s(b)\delta_{st}$, and so $A\rtimes S$ is the familiar skew group ring. Notice also that the quotient of $\mathcal{L}$ by any ideal yields a system, in the sense of \cite[Pg 2]{Nystedt1}.

A key class of inverse semigroup actions is that of spectral actions. Recall from \cite{BenDan} that an action $\alpha$ of an inverse semigroup $S$ on $A$ is called \emph{spectral} if it is non-degenerate and $D_e$ has a unit element $1_e$ for each $e\in E(S)$, which is necessarily a central idempotent of $A$. In the case of a spectral action, the central idempotents in the Boolean algebra generated by the $1_e$ are a set of local units for $A$.

\subsection{The interplay between groupoid convolution algebras with coefficients in a sheaf of rings and skew inverse semigroup rings.}\label{sec:interplay}

It is shown in \cite{BenDan} that, under certain mild conditions, groupoid convolution algebras with coefficients in a sheaf of rings can be realized as skew inverse semigroup rings and vice versa. We recall this more precisely below.

We denote by $\mathscr G^a$ the set of all compact open bisections of $\mathscr G$. It is an inverse semigroup with operations given by \[BC = \{ bc \in \mathscr G \mid b\in B, c\in C, \text{ and } \dom(b) = \ran(c)\}\] and $B^*=\{b^{-1}\mid b\in B\}$. 

Let $\mathcal O=(E,p,\alpha)$ be a $\mathscr G$-sheaf of rings.  Put $A=\Gamma_c(\mathscr G\skel 0,\mathcal O)$, which we view as the ring of compactly supported continuous sections of $\mathcal O$ over $\mathscr G\skel 0$ with pointwise operations. In \cite{BenDan} a spectral action of $S=\grp^a$ on $A$ is defined. 
If $U\subseteq \mathscr G\skel 0$ is open then we put $\mathcal O|_U = (p\inv(U),+,\cdot)$; it is a sheaf of rings on $U$.  Then $A(U)=\Gamma_c(U,\mathcal O|_U)$ can be identified with the subring of $A$ consisting of sections supported on $U$. For $s\in S$, put $D_s = A(\ran(s))$.  Note that since $\ran(s)$ is compact open, $D_s$ has an identity, the mapping $\chi_{\ran(s)}$.   We define an isomorphism $\til \alpha_s\colon D_{s^*}\to D_s$ by
\[\til \alpha_s(f)(\ran(\gamma)) = \alpha_{\gamma}(f(\dom(\gamma)))\] for $\gamma\in s$ and $f\in D_{s^*}$. We then have the following (which is proved more generally for inverse subsemigroups of $\grp^a$ satisfying certain conditions).

\begin{Thm}[\!{\cite[Theorem~7.1]{BenDan}}]\label{siri}
Let $\mathcal O$ be a $\mathscr G$-sheaf of rings on an ample groupoid $\mathscr G$. Then, \[\Gamma_c(\mathscr G,\mathcal O)\cong \Gamma_c(\mathscr G\skel 0,\mathcal O)\rtimes \mathscr G^a\]
as rings.
\end{Thm}

To obtain the converse characterization, we recall some key concepts first.

Given a Boolean action $\rho$ of an inverse semigroup $S$ on a space $X$, the \emph{groupoid of germs} $\mathscr G=S\ltimes X$ is defined as follows (see \cite{Steinbergalgebra}).  The unit space $\mathscr G\skel 0$ is taken to be $X$ and \[\mathscr G\skel 1 = \{(s,x)\in S\times X\mid x\in D_{s^*}\}/{\sim}\] where $(s,x)\sim (t,y)$ if and only if $x=y$ and there exists $u\leq s,t$ with $x\in D_{u^*}$.  We write $[s,x]$ for the class of $(s,x)$. The source and range maps are defined by $\dom([s,x]) = x$ and $\ran([s,x]) = \rho_s(x)$, the product is defined by $[s,\rho_t(x)][t,x] = [st,x]$, and the inverse is given by $[s,x]\inv = [s^*,\rho_s(x)]$.  The topology on $\mathscr G\skel 0$ is that of $X$, whereas a basis of neighborhoods for $\mathscr G\skel 1$ is given by the sets $(s,U)$, where $U\subseteq D_{s^*}$ is open (compact open if $X$ is zero-dimensional), and $(s,U)=\{[s,x]\mid x\in U\}.$

We also need to recall the generalized Stone space of a generalized Boolean algebra $B$.  A \emph{character} of $B$ is a non-zero Boolean algebra homomorphism $\lambda\colon B\to \{0,1\}$ to the two-element Boolean algebra.  
The (generalized) \emph{Stone space} of $B$ is the space $\wh {B}$ of characters of $B$ topologized by taking as a basis the sets
\[D(a)=\{\lambda\in \wh{B}\mid \lambda(a)=1\}\] with $a\in B$.

Let $\alpha$ be a spectral action of an inverse semigroup $S$ on a ring $A$.
We want to define a Boolean action of $S$ on the Pierce spectrum $\wh A$ of $A$ (recall from \cite{BenDan} that the Pierce spectrum of $A$ is $ \wh {B}$, where 
$B=E(Z(A))$). 

For $s\in S$, let \[\wh D_{s} = \{\lambda\in \wh A\mid \lambda(1_{ss^*})=1\}.\]   Notice that $\wh D_s$ is compact open and can be identified with the Pierce spectrum of $D_s$.   Define $\wh \alpha_s\colon \wh D_{s^*}\to \wh D_s$ by \[\wh \alpha_s(\lambda)(e) = \lambda (\alpha_{s^*}(e1_{ss^*}))\] for $e\in E(Z(A))$.

Let $\mathscr G=S\ltimes \wh A$ be the corresponding groupoid of germs.  It is ample as $\wh A$ has a basis of compact open sets. We want to define a $\mathscr G$-sheaf of $\mathcal O_A$ rings such that $A\rtimes S\cong \Gamma_c(\mathscr G,\mathcal O_A)$. For this, 
$\mathcal O_A$ is the sheaf of unital rings constructed as follows:
The underlying space of $\mathcal O_A$ is defined to be $E=\coprod_{\lambda\in \wh{E(Z(A))}} A/I_{\lambda}$, where \[I_{\lambda} = \{a\in A\mid \exists e\in \lambda\inv(1)\ \text{with}\ ea=0\}\] and the class $a+I_{\lambda}$ is denoted $[a]_{\lambda}$. If $r\in A$ and $e\in E(Z(A))$, then we put
\[(r,D(e)) = \{[r]_{\lambda}\mid \lambda \in D(e)\}.\]   The sets of the form $(r,D(e))$ form a basis for a topology on $E$ and $p\colon E\to \wh{B}$, defined by $p([r]_{\lambda}) = \lambda$, maps $(r,D(e))$ homeomorphically to $D(e)$, whence $p$ is a local homeomorphism.  The ring structures on the $A/I_{\lambda}$ turn $\mathcal O_A=(E,p,+,\cdot)$ into a sheaf of unital rings on $\wh{E(Z(A))}$. Each stalk of $O_A$ is $\mathcal O_{A,\lambda} = A/I_{\lambda}$, and the unit is the class $[e]_{\lambda}$ where $\lambda(e)=1$.  We obtain the $\mathscr G$-sheaf structure on $\mathcal O_A$ by putting
\begin{equation}\label{eq:sheaf.action.gpd.germs}
\alpha_{[s,\lambda]}([a]_{\lambda}) = [\alpha_s(1_{s^*s}a)]_{\wh \alpha_s(\lambda)}
\end{equation}
for $[s,\lambda]\in\mathscr G\skel 1$.  

\begin{Thm}[\!\!{\cite[Theorem~9.5]{BenDan}}]\label{t:sheaf.Pierce}
If $S$ is an inverse semigroup with a spectral action $\alpha$ on a ring $A$, then $A\rtimes S\cong \Gamma_c(S\ltimes \wh A,\mathcal O_A)$ with the $S\ltimes \wh A$-sheaf  structure on $\mathcal O_A$ coming from \eqref{eq:sheaf.action.gpd.germs} and the usual sheaf of rings structure on $\mathcal O_A$ over the Pierce spectrum $\wh A$.
\end{Thm}


\section{Induced modules}\label{induced}

In this section, we describe the induction of modules from isotropy skew group rings for groupoid algebras with coefficients in a sheaf, and we show how we can use induction to study the ideals of this algebra. We fix an ample groupoid $\grp$ and a $\grp$-sheaf of rings $\OO$. 

Given $x\in\grp^{(0)}$, we define $L_x=\{\gamma\in\grp\mid \dom(\gamma)=x\}$ and $\LLL_x=\bigoplus_{\gamma\in L_x}\OO_{\ran(\gamma)}$ as a direct sum of abelian groups. For $\gamma\in L_x$, $1_{\gamma}\in\LLL_x$ represents the element that is $1_{\ran(\gamma)}$ at coordinate $\gamma$ and it is $0_{\ran(\eta)}$ at all other coordinates $\eta\in L_x$ with $\eta\neq\gamma$. Since the isotropy group $\grp^x_x$ of $x$ acts on $\OO_x$, we can form the skew group ring \[B_x:=\OO_x\rtimes \grp^x_x\] which we call the \emph{isotropy skew group ring} at $x$.

\begin{Prop}\label{prop:LLL_x.right.free.module}
The abelian group $\LLL_x$ has a structure of free right $B_x$-module where the right action is given by
\begin{equation}\label{eq:LLL_x.right.module}
a1_\gamma \cdot b\delta = a\alpha_{\gamma}(b)1_{\gamma\delta}
\end{equation}
for $\gamma\in L_x$, $a\in\OO_{\ran(\gamma)}$, $b\in\OO_x$ and $\delta\in\grp^x_x$, and extended in the natural way. Moreover, if for each $y\in \mathrm{Orb}(x)$, we choose $\eta_y\in\grp$ such that $\dom(\eta_y)=x$ and $\ran(\eta_y)=y$, then $\{1_{\eta_y}\}_{y\in \mathrm{Orb}(x)}$ is a basis for $\LLL_x$.
\end{Prop}

\begin{proof}
We check the associativity of the right action, the other properties being straightforward. For that, consider $\gamma\in L_x$, $a\in\OO_{\ran(\gamma)}$, $b,c\in\OO_x$ and $\delta,\eta\in\grp^x_x$. We have that
\[(a1_\gamma\cdot b\delta)\cdot c\eta=a\alpha_\gamma(b)1_{\gamma\delta}\cdot c\eta=a\alpha_\gamma(b)\alpha_{\gamma\delta}(c)1_{\gamma\delta\eta}=\]
\[a\alpha_\gamma(b\alpha_\delta(c))1_{\gamma\delta\eta}=a1_\gamma\cdot b\alpha_{\delta}(c)\delta\eta=a1_\gamma\cdot(b\delta c\eta). \]

To prove the second part of the statement, for each $y\in \mathrm{Orb}(x)$, we fix an element $\eta_y\in\grp$ such that $\dom(\eta_y)=x$ and $\ran(\eta_y)=y$. In order to show that $\LLL_x$ is free we consider the abelian group homomorphism $\Phi:\LLL_x\to\bigoplus_{y\in \mathrm{Orb}(x)}B_x1_y$ given by
\begin{equation}\label{eq:LLL_x.free.module.iso}
    \Phi(a1_\gamma)=\alpha_{\eta_{\ran(\gamma)}\inv}(a)\eta_{\ran(\gamma)}\inv\gamma1_{\ran(\gamma)},
\end{equation}
where $1_y$ is the canonical vector for $y\in \mathrm{Orb}(x)$, and $\bigoplus_{y\in \mathrm{Orb}(x)}B_x1_y$ has the natural right $B_x$-module structure. Observe that $\alpha_{\eta_{\ran(\gamma)}\inv}(a)\in\OO_{\dom(\eta_{\ran(\gamma)})}=\OO_x$ and $\dom(\eta_{\ran(\gamma)}\inv\gamma)=\ran(\eta_{\ran(\gamma)}\inv\gamma)=x$, so that $\Phi$ is indeed well-defined. We now show that $\Phi$ is a right $B_x$-module isomorphism.

Given $\gamma\in L_x$, $a\in\OO_{\ran(\gamma)}$, $b\in\OO_x$ and $\delta\in\grp^x_x$, we have that
\[\Phi(a1_\gamma\cdot b\delta)=\Phi(a\alpha_{\gamma}(b)1_{\gamma\delta})=\alpha_{\eta_{\ran(\gamma\delta)}\inv}(a\alpha_\gamma(b))\eta_{\ran(\gamma\delta)}\inv\gamma\delta1_{\ran(\gamma\delta)}=\]
\[\alpha_{\eta_{\ran(\gamma)}\inv}(a)\alpha_{\eta_{\ran(\gamma)}\inv\gamma}(b)\eta_{\ran(\gamma)}\inv\gamma\delta1_{\ran(\gamma)}=\alpha_{\eta_{\ran(\gamma)}\inv}(a)\eta_{\ran(\gamma)}\inv\gamma b\delta1_{\ran(\gamma)}=\]
\[(\alpha_{\eta_{\ran(\gamma)}\inv}(a)\eta_{\ran(\gamma)}\inv\gamma 1_{\ran(\gamma)})b\delta=\Phi(a1_\gamma)b\delta,\]
so that $\Phi$ is a right $B_x$-module homomorphism.

In order to build the inverse of $\Phi$ we use the following identification as abelian groups, $\bigoplus_{y\in \mathrm{Orb}(x)}B_x1_y=\bigoplus_{y\in \mathrm{Orb}(x)}\bigoplus_{\delta\in\grp^x_x}\OO_x\delta1_y$. We now define the abelian group homomorphism $\Psi:\bigoplus_{y\in \mathrm{Orb}(x)}B_x1_y\to\LLL_x$ by
\[\Psi(b\delta1_y)=\alpha_{\eta_y}(b)1_{\eta_y\delta},\]
where $b\in\OO_x$, $\delta\in\grp^x_x$ and $y\in \mathrm{Orb}(x)$. Here, $\eta_y\delta\in L_x$ and $\alpha_{\eta_y}(b)\in\ran(\eta_y)=\ran(\eta_y\delta)$ so that $\Psi$ is well-defined.

We now prove that $\Psi=\Phi\inv$. Given $\gamma\in L_x$ and $a\in\OO_{\ran(\gamma)}$, we have that
\[\Psi(\Phi(a1_\gamma))=\Psi(\alpha_{\eta_{\ran(\gamma)}\inv}(a)\eta_{\ran(\gamma)}\inv\gamma1_{\ran(\gamma)})=\alpha_{\eta_{\ran(\gamma)}}(\alpha_{\eta_{\ran(\gamma)}\inv}(a))1_{\eta_{\ran(\gamma)}\eta_{\ran(\gamma)}\inv\gamma}=a1_\gamma.\]
On the other hand, given $b\in\OO_x$, $\delta\in\grp^x_x$ and $y\in \mathrm{Orb}(x)$, we have that
\[\begin{array}{ll} \Phi(\Psi(b\delta1_y))& =\Phi(\alpha_{\eta_y}(b)1_{\eta_y\delta})=\alpha_{\eta_{\ran(\eta_y\delta)}\inv}(\alpha_{\eta_y}(b))\eta_{\ran(\eta_y\delta)}\inv\eta_y\delta1_{\ran(\eta_y\delta)} \\ \\ & = \alpha^{-1}_{\eta_{y}}(\alpha_{\eta_y}(b))\eta_{y}\inv\eta_y\delta1_{y}=b\delta1_y.\end{array}\]

It follows that $\Phi$ is a right $B_x$-module isomorphism between $\LLL_x$ and the free module $\bigoplus_{y\in \mathrm{Orb}(x)}B_x1_y$. Moreover, for each $y\in \mathrm{Orb}(x)$, $\Psi(1_y)=1_{\eta_y}$ so that $\{1_{\eta_y}\}_{y\in \mathrm{Orb}(x)}$ is a basis for $\LLL_x$.
\end{proof}

Let $x\in\grp^{(0)}$. From now on we assume that, for each $y\in \mathrm{Orb}(x)$, we have chosen $\eta_y\in\grp$ such that $\eta_x=x$ and, for $y\neq x$, \begin{equation}\label{basis} \dom(\eta_y)=x \text{ and } \ran(\eta_y)=y.
\end{equation}

\begin{Prop}\label{prop:LLL_x.bimodule}
The abelian group $\LLL_x$ has a structure of $\Gamma_c(\mathscr G,\mathcal O)$-left module with left action given by
\begin{equation}\label{eq:LLL_x.left.module}
    f\cdot a1_\gamma=\sum_{\dom(\beta)=\ran(\gamma)}f(\beta)\alpha_{\beta}(a)1_{\beta\gamma},
\end{equation}
where $f\in \Gamma_c(\mathscr G,\mathcal O)$, $\gamma\in L_x$ and $a\in\OO_{\ran(\gamma)}$. Moreover, with the $B_x$ right action given by Equation \eqref{eq:LLL_x.right.module}, we have that $\LLL_x$ is a $\Gamma_c(\mathscr G,\mathcal O)$-$B_x$-bimodule.
\end{Prop}

\begin{proof}
By Proposition \ref{p:finiteness.prop}, the sum in \eqref{eq:LLL_x.left.module} is finite. Also notice that $f(\beta),\alpha_{\beta}(a)\in\OO_{\ran(\beta)}=\OO_{\ran(\beta\gamma)}$, so that the expression in \eqref{eq:LLL_x.left.module} is well-defined. For the left-module properties, we prove the associativity, the other properties being immediate. For $f,g\in \Gamma_c(\mathscr G,\mathcal O)$, $\gamma\in L_x$ and $a\in\OO_{\ran(\gamma)}$, we have that
\begin{align*}
    f\cdot(g\cdot a1_\gamma) & =f\cdot\left(\sum_{\dom(\beta)=\ran(\gamma)}g(\beta)\alpha_{\beta}(a)1_{\beta\gamma}\right) \\
    & =\sum_{\dom(\beta)=\ran(\gamma)}\sum_{\dom(\delta)=\ran(\beta)}f(\delta)\alpha_{\delta}(g(\beta)\alpha_{\beta}(a))1_{\delta\beta\gamma} \\
    & =\sum_{\dom(\beta)=\ran(\gamma)}\sum_{\dom(\delta)=\ran(\beta)}f(\delta)\alpha_{\delta}(g(\beta))\alpha_{\delta\beta}(a)1_{\delta\beta\gamma} \\
    & =\sum_{\dom(\zeta)=\ran(\gamma)}\sum_{\delta \beta=\zeta}f(\delta)\alpha_{\delta}(g(\beta))\alpha_{\delta\beta}(a)1_{\delta\beta\gamma} \\
    & =\sum_{\dom(\zeta)=\ran(\gamma)}(f*g)(\zeta)\alpha_{\zeta}(a)1_{\zeta\gamma}\\
    &=(f*g)\cdot a1_\gamma.
\end{align*}

We now check the associativity with respect to the left and right actions. Fix $f\in \Gamma_c(\mathscr G,\mathcal O)$, $\gamma\in L_x$, $a\in\OO_{\ran(\gamma)}$, $b\in\OO_x$ and $\delta\in\grp^x_x$. Then,
\begin{align*}
    (f\cdot a1_\gamma)\cdot b\delta & = \left(\sum_{\dom(\beta)=\ran(\gamma)}f(\beta)\alpha_{\beta}(a)1_{\beta\gamma}\right)\cdot b\delta \\
    &=\sum_{\dom(\beta)=\ran(\gamma)}f(\beta)\alpha_{\beta}(a)\alpha_{\beta\gamma}(b)1_{\beta\gamma\delta} \\
    &=\sum_{\dom(\beta)=\ran(\gamma)}f(\beta)\alpha_{\beta}(a\alpha_{\gamma}(b))1_{\beta\gamma\delta}\\
    &=f\cdot(a\alpha_{\gamma}(b)1_{\gamma\delta})\\
    &=f\cdot(a1\gamma\cdot b\delta).
\end{align*}
\end{proof}

\begin{Rmk}
A sum such as in \eqref{eq:LLL_x.left.module} can be decomposed using orbits. In general, for $x\in\grp^{(0)}$ the sum $\sum_{\dom(\beta)=x}$ can be decomposed as $\sum_{y\in \mathrm{Orb}(x)}\sum_{\gamma:x\to y}$, where the last sum is over all $\gamma\in\grp$ such that $\dom(\gamma)=x$ and $\ran(\gamma)=y$.
\end{Rmk}

In possession of the bimodule $\LLL_x$, we can define an induction functor  $\Ind_x:\module{B_x}\to\module{\Gamma_c(\mathscr G,\mathcal O)}$ defined by \[\Ind_x(M)=\LLL_x\otimes_{B_x} M.\] The induction functor is exact as $\LLL_x$ is a free right $B_x$-module. We can also use the basis of Proposition~\ref{prop:LLL_x.right.free.module} to decompose $\Ind_x(M)$ as a direct sum. More specifically, for each $y\in \mathrm{Orb}(x)$ let $\eta_y\in\grp$ as in (\ref{basis}). Then, given $M\in \module{B_x}$, we have that
\begin{equation}\label{eq:decomposition.ind.M}
    \Ind_x(M)=\LLL_x\otimes_{B_x} M=\bigoplus_{y\in \mathrm{Orb}(x)}1_{\eta_y}\otimes M.
\end{equation}
We now compute how the $\Gamma_c(\mathscr G,\mathcal O)$ action behaves with respect to the isomorphism decomposition of $\Ind_x(M)$. Given $f\in \Gamma_c(\mathscr G,\mathcal O)$, $y\in \mathrm{Orb}(x)$ and $m\in M$, we have that
\begin{align}\label{eq:Gamma.acts.induced.module}
    f\cdot 1_{\eta_y}\otimes m & = \sum_{z\in \mathrm{Orb}(x)}\sum_{\zeta:y\to z}f(\zeta)1_{\zeta\eta_y}\otimes m \nonumber  \\
    &=\sum_{z\in \mathrm{Orb}(x)}\sum_{\zeta:y\to z}\alpha_{\eta_z\eta_z\inv}(f(\zeta))1_{\eta_z\eta_z\inv\zeta\eta_y}\otimes m \nonumber \\
    &=\sum_{z\in \mathrm{Orb}(x)}\sum_{\zeta:y\to z}1_{\eta_z}\otimes\alpha_{\eta_z\inv}(f(\zeta))\eta_z\inv\zeta\eta_y \cdot m,
\end{align}
where the first equality follows from \eqref{eq:LLL_x.left.module} and the decomposition of the set $\{\beta\in\grp\mid\dom(\beta)=y\}$ as $\bigcup_{z\in \mathrm{Orb}(x)}\{\zeta\in\grp\mid \dom(\zeta)=y,\ran(\zeta)=z\}$.

The following lemma follows immediately from Equation~\eqref{eq:Gamma.acts.induced.module}.

\begin{Lemma}\label{lem:ann.ind.module}
Let $M$ be a $\module{B_x}$, $\Ind_x(M)$ be the induced module, and for each $y\in \mathrm{Orb}(x)$ let $\eta_y\in\grp$ as in (\ref{basis}). Then, for every $f\in \Gamma_c(\mathscr G,\mathcal O)$, we have that $f\in\Ann(\Ind_x(M))$ if, and only if, for every $y,z\in \mathrm{Orb}(x)$, we have that $\sum_{\zeta:y\to z}\alpha_{\eta_z\inv}(f(\zeta))\eta_z\inv\zeta\eta_y\in\Ann(M)$.
\end{Lemma}

Now let $\mathcal M$ be a $\mathscr G$-sheaf of $\mathcal O$-modules. Our next goal is to describe $\Ann(\Gamma_c(\mathscr G,\mathcal M))$ in terms of the annihilators of the induced modules. First, we have to give a structure of left $B_x$-module to $\mathcal M_x$ for each $x\in\grp^{(0)}$. Given $b\in\OO_x$, $\delta\in\grp^x_x$ and $m\in\mathcal M_x$, we define
\begin{equation}\label{eq:M_x.left.B_x.module}
    b\delta\cdot m:=b \beta_\delta(m).
\end{equation}
As before, we only prove the associativity of the left product. Suppose we are given as well, $c\in\OO_x$ and $\eta\in\grp^x_x$, then
\[b\delta\cdot (c\eta\cdot m)=b\delta\cdot (c\beta_{\eta}(m))=b\beta_{\delta}(c\beta_{\eta}(m))=b\alpha_\delta(c)\beta_{\delta\eta}(m)=\]
\[(b\alpha_\delta(c)\delta\eta)\cdot m=(b\delta c\eta)\cdot m.\]

The following generalizes \cite[Theorem 5]{Ben}. We adapt the proof there to accommodate the extra structure.

\begin{Thm}\label{thm:ann.ideal}
Let $\mathcal M$ be a $\mathscr G$-sheaf of $\mathcal O$-modules. Then,
\[\Ann(\Gamma_c(\mathscr G,\mathcal M))=\bigcap_{x\in\grp^{(0)}} \Ann(\Ind_x(\mathcal M_x)).\]
Consequently, every ideal of $\Gamma_c(\mathscr G,\mathcal O)$ is an intersection of annihilators of induced modules.
\end{Thm}

\begin{proof}
Let $f\in \Ann(\Gamma_c(\mathscr G,\mathcal M))$ and $x\in\grp^{(0)}$. In order to prove that $f\in \Ann(\Ind_x(\mathcal M_x))$, by Lemma \ref{lem:ann.ind.module}, it suffices to show that for every $y,z\in \mathrm{Orb}(x)$, we have that $\sum_{\zeta:y\to z}\alpha_{\eta_z\inv}(f(\zeta))\eta_z\inv\zeta\eta_y\in\Ann(\mathcal M_x)$. Let $m\in\mathcal M_x$. Then, by Proposition~\ref{p:lots.of.secs}, there exists $s\in \Gamma_c(\mathscr G,\mathcal M)$ such that $s(x)=m$. Fix $y,z\in \mathrm{Orb}(x)$. By Proposition \ref{p:finiteness.prop}, $|\ran\inv(z)\cap\supp(f)|<\infty$, and since $\grp^{(0)}$ is Hausdorff, there exists a compact-open neighborhood $U$ of $y$ such that $U\cap \dom(\ran\inv(z)\cap\supp(f))\scj\{y\}$. Now, let $U_y$ and $U_z$ be compact-open bisections such that $\eta_y\in U_y$ and $\eta_z\in U_z$. By replacing $U_y$ with $UU_y$, if necessary, we may assume that $\ran(U_y)\scj U$. Since $f\in \Ann(\Gamma_c(\mathscr G,\mathcal M))$ and the annihilator is an ideal, we have that $\chi_{U_z\inv}*f*\chi_{U_y}\in \Ann(\Gamma_c(\mathscr G,\mathcal M))$, and therefore
\begin{equation}\label{eq:ann.1}
    0=((\chi_{U_z\inv}*f*\chi_{U_y})s)(x)=\sum_{\gamma\in\ran\inv(x)}(\chi_{U_z\inv}*f*\chi_{U_y})(\gamma)\beta_{\gamma}(s(\dom(\gamma))).
\end{equation}
Notice that, for each $\gamma\in\ran\inv(x)$
\[(\chi_{U_z\inv}*f*\chi_{U_y})(\gamma)=\sum_{\gamma_1\gamma_2\gamma_3=\gamma}\chi_{U_z\inv}(\gamma_1)\alpha_{\gamma_1}(f(\gamma_2))\alpha_{\gamma_1\gamma_2}(\chi_{U_y}(\gamma_3)).\]
We are interested in the non-zero terms of the above sum. For that, it is necessary that $\gamma_1=\eta_z\inv$, since $\eta_z\inv$ is the sole element of $U_z\inv$ such that the range is $x$. Also, by the properties of $U$ and $U_y$, it is necessary that $\dom(\gamma_2)=y$ and hence $\gamma_3=\eta_y$. It follows that
\begin{equation}\label{eq:ann.2}
    (\chi_{U_z\inv}*f*\chi_{U_y})(\gamma)=\sum_{\substack{\zeta:y\to z \\ \gamma=\eta_z\inv\zeta\eta_y}}\alpha_{\eta_z\inv}(f(\zeta)).
\end{equation}
By substituting \eqref{eq:ann.2} in \eqref{eq:ann.1} and making the appropriate change of variables, we get
\begin{align*}
    0&=\sum_{\zeta:y\to z}\alpha_{\eta_z\inv}(f(\zeta))\beta_{\eta_z\inv\zeta\eta_y}(s(x)) \\
    &=\sum_{\zeta:y\to z}\alpha_{\eta_z\inv}(f(\zeta))\eta_z\inv\zeta\eta_y\cdot m.
\end{align*}
Since $m$ was arbitrary, $\sum_{\zeta:y\to z}\alpha_{\eta_z\inv}(f(\zeta))\eta_z\inv\zeta\eta_y\in\Ann(\mathcal M_x)$ as we wanted to prove.

Suppose now that $f\in\Ann(\Ind_x(\mathcal M_x))$ for every $x\in\grp^{(0)}$. Given $s\in\secmod$, we want to prove that $fs=0$. For $x\in\grp^{(0)}$, by Lemma~\ref{lem:ann.ind.module}, if $y\in \mathrm{Orb}(x)$  and $\eta_y\in\grp$ is as in (\ref{basis}) then, we have that
\begin{align*}
    0 & =\sum_{\gamma:x\to y}\alpha_{\eta_y\inv}(f(\gamma))\eta_y\inv\gamma\cdot s(x) \\
    &=\sum_{\gamma:x\to y}\alpha_{\eta_y\inv}(f(\gamma))\beta_{\eta_y\inv\gamma}(s(x)) \\
    &=\beta_{\eta_y\inv}\left(\sum_{\gamma:x\to y} f(\gamma)\beta_{\gamma}(s(x))\right),
\end{align*}
and since $\beta_{\eta_y\inv}$ is injective, we conclude that $\sum_{\gamma:x\to y} f(\gamma)\beta_{\gamma}(s(x))=0$. This implies that if we fix $y\in\grp^{(0)}$,
\[fs(y)=\sum_{x\in \mathrm{Orb}(y)}\sum_{\gamma:x\to y}f(\gamma)\beta_{\gamma}(s(x))=0,\]
and hence $f\in\Ann(\secmod)$.

The last part follows from Theorem \ref{t:disint} and Remark~ \ref{r:find.ideals}.
\end{proof}

\section{Applications}\label{s:applications}

\subsection{Primitive ideals} In this section, we prove that every primitive ideal of $\secring$ is the annihilator of a single induced representation. Recall that an ideal is \emph{(left) primitive} if it is the annihilator of a simple (left) module.

\begin{Thm}\label{thm:primitive.ideal}
Let $\mathscr G$ be an ample groupoid and $\mathcal O$ a $\mathscr G$-sheaf of rings. If $I$ is a primitive ideal of $\secring$, then there exists $x\in\grp^{(0)}$ and $M$ a left $B_x$-module such that $I=\Ann(\Ind_x(M))$.
\end{Thm}

\begin{proof}
By Theorem~\ref{thm:ann.ideal} and Remark~\ref{r:find.ideals}, there exists $\mathcal M$ a $\mathscr G$-sheaf of $\mathcal O$-modules such that $I=\Ann(\Gamma_c(\mathscr G,\mathcal M))=\bigcap_{y\in\grp^{(0)}} \Ann(\Ind_y(\mathcal M_y))$ and $\secmod$ is a simple $\secring$-module. The latter implies that there exists $x\in\grp^{(0)}$ such that $\mathcal M_x\neq 0$ . We take $M=\mathcal M_x$ with the left $B_x$-module structure given by \eqref{eq:M_x.left.B_x.module}. Clearly $I\scj \Ann(\Ind_x(M))$.


Suppose now that $I\neq\Ann(\Ind_x(M))=:J_x$. Then there exists $s\in\secmod$ such that $J_xs\neq 0$. Since $\secmod$ is simple and $J_xs$ is a non-trivial submodule, we have that $J_xs=\secmod$. Consider now $m\in M\setminus\{0\}$ and let $t\in\secmod$ be such that $t(x)=m$ (such $t$ exists by Proposition~\ref{p:lots.of.secs}). Then there exists $f\in J_x$ such that $t=fs$, and therefore,
\[m=t(x)=fs(x)=\sum_{y\in \mathrm{Orb}(x)}\sum_{\gamma:y\to x} f(\gamma)\beta_{\gamma}(s(y)).\]
On the other hand, since $f\in J_x$ and for any $y\in \mathrm{Orb}(x)$ we have $\beta_{\eta_y^{ -1}}(s(y))\in \mathcal M _x$, using \eqref{eq:M_x.left.B_x.module} and Lemma \ref{lem:ann.ind.module} (and recalling that $\eta_x=x$), we obtain that
\[0=\sum_{\gamma:y\to x}f(\gamma)\gamma\eta_y\cdot\beta_{\eta_y^{ -1}}(s(y))=\sum_{\gamma:y\to x} f(\gamma)\beta_{\gamma}(s(y)).\]
This way,
\[0\neq m=\sum_{y\in \mathrm{Orb}(x)} 0=0,\]
which is a contradiction.
\end{proof}

\subsection{Simple modules} 
We now prove that the induction of modules preserves simplicity.

\begin{Thm}\label{thm:simple.module}
Let $\mathscr G$ be an ample groupoid, $\mathcal O$ a $\mathscr G$-sheaf of rings and $x\in\grp^{(0)}$. If $M$ is a simple $B_x$-module, then $\Ind_x(M)$ is a simple $\secring$-module.
\end{Thm}

\begin{proof}
For $w\in\Ind_x(M)\setminus\{0\}$, we have to prove that $\secring\cdot w=\Ind_x(M)$. Due to the decomposition given in \eqref{eq:decomposition.ind.M}, we can write $w=\sum_{i=1}^k 1_{\eta_{y_i}}\otimes m_i$, where $y_i\in \mathrm{Orb}(x)$ for all $i$, $y_i\neq y_j$ if $i\neq j$, and $m_i\in M\setminus\{0\}$ for all $i$. Since $\grp^{(0)}$ is Hausdorff, we can find $U\scj\grp^{(0)}$ open set containing only $y_1$ among the $y_i$. Using Equation \eqref{eq:Gamma.acts.induced.module}, we can check that $\chi_U\cdot w=1_{\eta_{y_1}}\otimes m_1$. Then, we may assume without loss of generality that $w=1_{\eta_y}\otimes m$ for some $y\in \mathrm{Orb}(x)$ and some $m\in M\setminus\{0\}$. Also, it is sufficient to prove that given $z\in \mathrm{Orb}(x)$ and $m'\in M$, there exists $f\in\secring$ such that $f\cdot w=1_{\eta_z}\otimes m'$.

We now fix $z\in \mathrm{Orb}(x)$ and $m'\in M$. Since $M$ is simple, there exists $a\in B_x$ such that $a\cdot m=m'$. We can write $a=\sum_{i=1}^k a_i\delta_i$, where $a_i\in\OO_x$ and $\delta_i\in\grp^x_x$ for all $i$. For each $i=1,\ldots,k$, we take $U_i$ a compact-open bisection containing $\delta_i$ and $s_i\in\secring$ such that $s_i(x)=a_i$, see Proposition~\ref{p:lots.of.secs}. We also take compact-open bisections $U_y$ and $U_z$ such that $\eta_y\in U_y$ and $\eta_z\in U_z$ respectively. Define $f=\sum_{i=1}^k\chi_{U_z}*s_i\chi_{U_i}*\chi_{U_y^{-1}}$. Arguing similarly to what was done in the proof or Theorem \ref{thm:ann.ideal} to prove Equation \eqref{eq:ann.2}, starting from \eqref{eq:Gamma.acts.induced.module}, we see that
\begin{align*}
    f\cdot 1_{\eta_y}\otimes m & =\sum_{v\in \mathrm{Orb}(x)}\sum_{\zeta:y\to v}1_{\eta_v}\otimes\alpha_{\eta_v\inv}(f(\zeta))\eta_v\inv\zeta\eta_y \cdot m  \\
    &=1_{\eta_z}\otimes\sum_{i=1}^k\sum_{\substack{\zeta:y\to z \\ \zeta=\eta_z\delta_i\eta_y\inv}}\alpha_{\eta_z\inv}(\chi_{U_z}*s_i\chi_{U_i}*\chi_{U_y^{-1}}(\zeta))\eta_z\inv\zeta\eta_y\cdot m  \\
    &=1_{\eta_z}\otimes\sum_{i=1}^k \alpha_{\eta_z\inv}(\alpha_{\eta_z}(s_i(\ran(\delta_i))))\delta_i\cdot m \\
    &=1_{\eta_z}\otimes\sum_{i=1}^k a_i\delta_i\cdot m \\
    &=1_{\eta_z}\otimes a\cdot m\\
    &=1_{\eta_z}\otimes m'.
\end{align*}
\end{proof}

We now show that the module $M$ in Theorem~\ref{thm:primitive.ideal} above can be chosen to be simple under some strong hypotheses on the rings $B_x$.  Let $J(S)$ denote the Jacobson radical of a ring $S$. A ring $S$ is called a \emph{left max ring} if each non-zero left $S$-module has a maximal (proper) submodule.  For example, any Artinian ring  $S$ is a left max ring.  Indeed, if $M\neq 0$, then $J(S)M\neq M$ by nilpotency of the Jacobson radical.  But $M/J(S)M$ is then a non-zero $S/J(S)$-module and every non-zero module over a semisimple ring is a direct sum of simple modules and hence has a simple quotient.  Thus $M$ has a maximal proper submodule.  A result of Hamsher~\cite{Hamsher} says that if $S$ is commutative, then $S$ is a left max ring if and only if $J(S)$ is $T$-nilpotent  (e.g., if $J(S)$ is nilpotent) and $S/J(S)$ is von Neumann regular ring. We now prove an analog of the Effros-Hahn conjecture for groupoid algebras with coefficients in a sheaf if all the isotropy skew group rings are left max rings.

\begin{Thm}\label{t:left.max}
Let $\mathscr G$ an ample groupoid and $\mathcal O$ a $\mathscr G$-sheaf of rings, with the isotropy skew group rings $B_x$ being left max rings for all $x\in \mathscr G^{(0)}$.  Then the primitive ideals of $\secring$ are exactly the ideals of the form $\Ann(\Ind_x(M))$, where $M$ is a simple $B_x$-module.
\end{Thm}
\begin{proof}
By Theorem~\ref{thm:simple.module} it suffices to show that any primitive ideal $I$ is of the form $\Ann(\Ind_x(M))$ where $M$ is a simple $B_x$-module.
By Theorem~\ref{t:disint} we may assume that our simple module with annihilator $I$ is of the form $\Gamma_c(\grp,\mathcal M)$ for some $\grp$-sheaf $\mathcal M$ of $\mathcal O$-modules. Let $x\in \mathscr G^0$ with $\mathcal M_x\neq 0$.  We already know from the proof of Theorem~\ref{thm:primitive.ideal} that $I = \Ann(\Ind_x(\mathcal M_x))$.  Let $N$ be a maximal $B_x$-submodule of $\mathcal M_x$ (which exists by assumption on $B_x$) and let $J=\Ann(\Ind_x(\mathcal M_x/N))$.  Since $\mathcal M_x/N$ is simple, it suffices to show that $J=I$.  Clearly, $I\subseteq J$ (by Lemma~\ref{lem:ann.ind.module}) since $\Ann(\mathcal M_u)\subseteq \Ann(\mathcal M_u/N)$.  So it suffices to show that $J$ annihilates $\Gamma_c(\grp,\mathcal M)$.

Suppose that this is not the case.  Then there exists $s\in \Gamma_c(\grp,\mathcal M)$ with $Js\neq 0$.  Since $J$ is an ideal and $\Gamma_c(\grp,\mathcal M)$ is simple, we deduce $Js=\Gamma_c(\grp,\mathcal M)$.   Let $m\in \mathcal M_x\setminus N$ (using that $N$ is a proper submodule) and let $t\in \Gamma_c(\grp,\mathcal M)$ with $t(x)=m$.  Then $t=fs$ with $f\in J$.  Let us compute
\begin{equation}\label{eq:big.helper.2}
m=t(x)=(fs)(x) = \sum_{v\in \mathcal O_x}\sum_{\gamma\colon v\to x}f(\gamma)\beta_{\gamma}(s(v)).
\end{equation}

Let us fix $v\in \mathcal O_x$, fix $\gamma_v\colon x\to v$ and set $\gamma_x=x$.  Then by the assumption $f\in J$ and Lemma~\ref{lem:ann.ind.module}, we have (since $\beta_{\gamma_v\inv}( s(v)\in \mathcal M_x$) that 
\[\sum_{\gamma\colon v\to x}f(\gamma)\beta_\gamma s(v)=\sum_{\gamma\colon v\to x} f(\gamma)\beta_{\gamma\gamma_v}(\beta_{\gamma_v\inv} s(v))\in N.\]  We deduce from \eqref{eq:big.helper.2} that $m\in N$, which is a contradiction. It follows that $J$ annihilates $\Gamma_c(\grp,\mathcal M)$ and so $I=J$.
\end{proof}

A result of Park shows that if $R$ is a unital ring and $G$ is a group acting on $R$ by automorphisms, then the skew group ring $R\rtimes G$ is Artinian if and only if $R$ is Artinian and $G$ is finite, see \cite{Park, NYSTEDT2018433}.  This leads to the following corollary to Theorem~\ref{t:left.max}.

\begin{Cor}\label{c:main}
Let $\grp$ be an ample groupoid with finite isotropy groups and let $\mathcal O$ be a $\grp$-sheaf of Artinian rings.     Then the primitive ideals of $\secring$ are precisely the annihilators of modules induced from simple modules of isotropy skew group rings.
\end{Cor}
\begin{proof}
Under these hypotheses each isotropy skew group ring is Artinian and hence left max.  The result follows from Theorem~\ref{t:left.max}.
\end{proof}

\subsection{von Neumann regularity of $\secringunit$}

In this section we prove that $\secringunit$ is von Neumann regular (a ring $R$ is von Neumann regular if $a\in aRa$ for all $a\in R$) if and only if $\OO$ is a sheaf of fields. One interesting consequence of this is that the zero section has closed image, which will be used in results studying the algebraic properties of $\secring$.

To show that when $\mathcal O$ is a sheaf of fields the zero section is closed we need first a sheaf theoretic result. Recall that $R^\times$ denotes the group of units of a ring $R$.

\begin{Lemma}\label{units:open}
Let $\mathcal O$ be a $\grp$-sheaf of rings. Then $\mathcal O^\times=\bigcup_{x\in \grp^{(0)}}\mathcal O_x^\times$ is open and the inversion map on $\mathcal O^\times$ is continuous.
\end{Lemma}
\begin{proof}
Let $r\in \mathcal O_x^\times$ with inverse $r'$.  Then there are neighborhoods $U,U'$ of $x$ in $\grp^0$ and sections $s\colon U\to \mathcal O$, $s'\colon U'\to \mathcal O$ with $s(x)=r$ and $s'(x)=r'$.  Without loss of generality, we may assume that $U=U'$, or else we can replace them by $U\cap U'$.  Since $r'r=rr'=1_x$, there is a neighborhood $W$ of $x$ with $W\subseteq U$ such that $(s'\ast s)|_W=(s\ast s')|_W=\chi_W$ and so $r\in s(W)\subseteq \mathcal O^\times$. As $s(W)$ is open we deduce $\mathcal O^\times$ is open. Note that on $s(W)$, inversion is given by $s'\circ p$ where $p\colon \mathcal O\to \grp^{(0)}$ is the projection.  Thus inversion is continuous.
\end{proof}

\begin{Cor}\label{c:zero.sec.closed}
Let $\mathcal O$ be a $\grp$-sheaf of fields.  Then the image of the zero section is closed.
\end{Cor}
\begin{proof}
The complement of the image of the zero section is $\mathcal O^\times$ and so the result follows from Lemma~\ref{units:open}.
\end{proof}

We also use Lemma~\ref{units:open} in the characterization of von Neumann's regularity of $\secringunit$, as we see below. Recall that a unital ring $R$ is indecomposable if it has no central idempotents except $0$ and $1$.  For example, fields and integral domains are indecomposable.  The special case of the following result for sheaves of rings on a compact totally disconnected space can be found in~\cite[Proposition~V.2.6]{johnstone}.

\begin{Prop}\label{p:vnr}
Let $\mathcal O$ be a $\grp$-sheaf of indecomposable commutative rings.  Then $\secringunit$ is von Neumann regular if and only if $\mathcal O_x$ is a field for all $x\in \grp^{(0)}$.
\end{Prop}
\begin{proof}
If $\mathcal O$ is a sheaf of fields and $s\in \secringunit$, then $U=\supp(s) = s^{-1}(\mathcal O^\times)$ is compact open (as $\mathcal O^\times$ is open by Lemma~\ref{units:open}) and $s\colon U\to \mathcal O^\times$.  By continuity of inversion on $\mathcal O^\times$, we deduce that $s'\colon U\to \mathcal O^\times$ given by $s'(x) =s(x)\inv$ is continuous and, by construction, it has support $U$. Thus $s'\in \secringunit$.  Trivially, $s\ast s'\ast s=s$ and so $\secringunit$ is von Neumann regular. 

Conversely, suppose that $\secringunit$ is von Neumann regular.  Let $0_x\neq r\in \mathcal O_x$.  Then there exists $s\in \secringunit$ with $s(x)=r$.  Choose $s'\in \secringunit$ with $s\ast s'\ast s=s$. Note that if $r'=s'(x)$, then $rr'r=r$ and so $rr'$ is a non-zero idempotent. But since $\mathcal O_x$ is indecomposable, this implies $rr'=1_x$.  Thus $r\in \mathcal O_x^\times$.
\end{proof}

Note that the Pierce sheaf associated to a commutative ring $A$ is a sheaf of indecomposable commutative rings over the Pierce spectrum of $A$ and hence the construction used in~\cite{BenDan} (see Section~\ref{sec:interplay}) to build a groupoid algebra with coefficients in a sheaf of ring from a skew inverse semigroups ring $A\rtimes S$ will produce a sheaf of fields when $A$ is commutative and von Neumann regular, see \cite{Pierce}.

\subsection{The centralizer of $\secringunit$ in $\secring$}

Maximal commutativity of $\secringunit$ will play a key role in describing primitivity and simplicity of $\secring$. It is also a much-studied concept in the theory of partial skew rings, see \cite{BGOR,GOR}. Of course, $\mathcal O$ should be a $\grp$-sheaf of commutative rings in order for $\secringunit$ to be commutative. In this section, we characterize (under suitable hypotheses on $\mathcal O$) the centralizer of $\secringunit$ in $\secring$, denoted by $C_{\secring}(\secringunit)$, and give sufficient and necessary conditions for $\secringunit$ to be maximal commutative. We also get a Generalized Uniqueness Theorem for $\secring$, similar to what is found for Steinberg algebras in \cite{CEP}.

We begin with an important observation that is true for any $\grp$-sheaf of rings (not necessarily commutative).

\begin{Prop}\label{p:centralizer.supp.iso}
Let $\grp$ be an ample groupoid and $\mathcal O$ a $\grp$-sheaf of rings.   If $f\in \secring$ centralizes $\secringunit$, then $\supp(f)\subseteq \mathrm{Iso}(\grp)$.  
In particular, if $\grp$ is Hausdorff and the zero section has closed image, then $f\in \Gamma_c(\mathrm{Int}(\mathrm{Iso}(\grp)),\mathcal O)$.
\end{Prop}
\begin{proof}
 If $\dom(\gamma)\neq \ran(\gamma)$ then, by the Hausdorff property of $\grp^{(0)}$, there is a compact open subset $U$ of $\grp^{(0)}$ with $\dom(\gamma)\in U$ and $\ran(\gamma)\notin U$.  Therefore $f(\gamma) = f\ast \chi_U(\gamma)=\chi_U\ast f(\gamma) = 0$. Thus, $f$ is supported on $\mathrm{Iso}(\grp)$.  If, in addition, $\grp$ is Hausdorff, then $f$ is continuous and hence, since the zero section is closed, $\supp(f)$ is open.  Thus $f\in \Gamma_c(\mathrm{Int}(\mathrm{Iso}(\grp)),\mathcal O).$
\end{proof}

We immediately obtain the following corollary.

\begin{Cor}\label{c:effective.case.max.cen}
Let $\grp$ be an effective Hausdorff ample groupoid and let $\mathcal O$ be a sheaf of commutative rings with the zero section closed.  Then $\secringunit$ is a maximal commutative subring of $\secring$.
\end{Cor}

\begin{Rmk}
Taking $\OO$ as the constant sheaf of Example~\ref{ex:constant.sheaf}, the above corollary recovers \cite[Proposition 3.8]{groupoidprimitive}.
\end{Rmk}

We will impose a few extra conditions on the sheaf $\mathcal O$ in order to obtain a characterization of the centralizer.
The first concept we need to introduce is the following.
 
\begin{Def} Let $\mathcal O$ be a $\grp$-sheaf of rings. We define $\ker \mathcal O = \{\gamma\in \mathrm{Iso}(\grp)\mid \alpha_\gamma(a)=a, \forall a\in \mathcal O_{\dom(\gamma)}\}$.  Note that $\ker \mathcal O$ is a subgroupoid of $\mathrm{Iso}(\grp)$ and $\mathrm{Int}(\ker\OO)$ is an ample subgroupoid. Therefore, we can consider the ring $\Gamma_c(\mathrm{Int}(\ker\OO),\OO)$, which can also be described as the additive subgroup of $\secring$ generated by $s\chi_U$, where $U$ is a compact-open bisection contained in $\mathrm{Int}(\ker\OO)$.
 \end{Def}

\begin{Lemma}\label{l:centralizer}
Let $\grp$ be an ample groupoid and $\mathcal O$ a $\grp$-sheaf of commutative rings. Then
\[\Gamma_c(\mathrm{Int}(\ker\OO),\OO)\scj \{f\in\secring\mid \supp(f)\scj\ker\OO\}\scj C_{\secring}(\secringunit).\]
If $\grp$ is Hausdorff and the zero section has closed image, then the first inclusion is an equality. Moreover, if $\mathcal O$ is $\grp$-sheaf of integral domains, then the second inclusion is an equality.
\end{Lemma}

\begin{proof}
The first inclusion is immediate. For the second inclusion, suppose first that $f\in\secring$ is such that $\supp(f)\scj\ker\OO$. Then, for $g\in\secringunit$, we have that
\[f*g(\gamma)=\begin{cases}
f(\gamma)\alpha_\gamma(g(\dom(\gamma))), & \text{if }\gamma\in \supp(f) \\
0, & \text{else,}
\end{cases}\]
and
\[g*f=\begin{cases}
g(\ran(\gamma))\alpha_{\ran(\gamma)}(f(\gamma)), & \text{if }\gamma\in \supp(f) \\
0, & \text{else.}
\end{cases}\]
For $\gamma\in \supp(f)\scj\ker\OO$, we have that
\begin{align*}
    f(\gamma)\alpha_\gamma(g(\dom(\gamma))) & =f(\gamma)g(\dom(\gamma)) \\ &=g(\ran(\gamma))f(\gamma)\\
    &=g(\ran(\gamma))\alpha_{\ran(\gamma)}(f(\gamma)),
\end{align*}
and hence $g*f=f*g$. Since $g$ was arbitrary, $f\in C_{\secring}(\secringunit)$.

When $\grp$ is Hausdorff, an element $f\in\secring$ is continuous. If, moreover the zero section has closed image, then $\supp(f)$ is an open subset of $\ker\OO$, from which we get the inclusion $\{f\in\secring\mid \supp(f)\scj\ker\OO\}\scj \Gamma_c(\mathrm{Int}(\ker\OO),\OO)$.

Now, suppose that $\mathcal O$ is $\grp$-sheaf of integral domains. For an element $f\in C_{\secring}(\secringunit)$, we claim that $f$ is supported on $\ker \mathcal O$. We already know that $f$ is supported on $\mathrm{Iso}(\grp)$ by Proposition~\ref{p:centralizer.supp.iso}. 
Suppose that $\gamma\in \mathrm{Iso}(\grp)$ and $f(\gamma)\neq 0$.  Put $x=\dom(\gamma)=\ran(\gamma)$.  For $a\in \mathcal O_x$, let $s$ be a section in $\secringunit$ with $s(x)=a$ (such a section exists by Proposition~\ref{p:lots.of.secs}).  Then,
\[f(\gamma)\alpha_{\gamma}(a) = f\ast s(\gamma) = s\ast f(\gamma) = af(\gamma).\]  Since $\mathcal O_x$ is an integral domain and $f(\gamma)\neq 0$, we deduce that $\alpha_{\gamma}(a)=a$.
\end{proof}


We now characterize maximal commutativity of $\secring$, without any assumption of Hausdorffness in the groupoid.

\begin{Prop}\label{p:neweffective}
Let $\grp$ be an ample groupoid and $\mathcal O$ a $\grp$-sheaf of commutative rings. If $\secringunit$ is a maximal commutative subring of $\secring$, then $\mathrm{Int}(\ker\OO)=\grp^{(0)}$. If moreover $\mathcal O$ is a $\grp$-sheaf of integral domains and $\Gamma_c(\mathrm{Int}(\ker\OO),\OO)=\{f\in\secring\mid \supp(f)\scj\ker\OO\}$, then the converse is also true.
\end{Prop}

\begin{proof}
Suppose first that $\secringunit$ is maximal commutative and let $U\subseteq \ker\mathcal O$ be compact open.  We must show that $U\subseteq \grp^{(0)}$.  
We have that $\chi_U$ centralizes $\secringunit$ by Lemma~\ref{l:centralizer} and hence, by maximality, $\chi_U\in \secringunit$.  Therefore,  $U\subseteq \grp^{(0)}$ and  $\mathrm{Int}(\ker\OO)=\grp^{(0)}$ as required. To get the equality $\Gamma_c(\mathrm{Int}(\ker\OO),\OO)=\{f\in\secring\mid \supp(f)\scj\ker\OO\}$, we observe that  $\secringunit\scj\Gamma_c(\mathrm{Int}(\ker \OO),\OO)\scj \{f\in\secring\mid \supp(f)\scj\ker\OO\}$ and apply Lemma~\ref{l:centralizer}.

The converse follows immediately from the hypothesis and Lemma \ref{l:centralizer}.
\end{proof}

\begin{Rmk} Recall that when $\mathcal O$ is the constant sheaf $\Delta(R)$ of Example~\ref{ex:constant.sheaf}, $\Gamma_c(\mathscr G,\Delta(R))$ is the usual groupoid algebra of \cite{Steinbergalgebra} and $\Gamma_c(\mathscr G^{(0)},\Delta(R))$ is the so called diagonal sub-algebra. In this case, we have that $\ker \Delta(R) = \mathrm{Iso}(\mathscr G)$, and so the condition $\mathrm{Int}(\ker \Delta(R))=\grp^{(0)}$ says that $\mathscr G$ is effective. We conclude that $\Gamma_c(\mathscr G^{(0)},\Delta(R))$ is maximal commutative in the usual groupoid algebra $\Gamma_c(\mathscr G,\Delta(R))$ if, and only if, $\mathscr G$ is effective and $\Gamma_c(\mathscr G^{(0)},\Delta(R)) =  \{f\in\Gamma_c(\mathscr G,\Delta(R))\mid \supp(f)\scj \mathrm{Iso}(\mathscr G) \}$.
\end{Rmk}

 In the case of Hausdorff groupoids, and sheaves such that the zero section is closed, the criteria given in Proposition~\ref{p:neweffective} for maximal commutativity can be simplified, as we show below. 

\begin{Cor}\label{cor:masa.Hausdorff}
Let $\grp$ be a Hausdorff ample groupoid and $\mathcal O$ a $\grp$-sheaf of integral domains such that the zero section is closed (eg. $\OO$ is a $\grp$-sheaf of fields).  Then, $\secringunit$ is a maximal commutative subring of $\secring$ if and only if $\mathrm{Int}(\ker \mathcal O) = \grp^{(0)}$.  In particular, if $\grp$ is effective, then $\secringunit$ is a maximal commutative subring.
\end{Cor}

\begin{proof}
When $\grp$ is Hausdorff, an element $f\in\secring$ is continuous. If, moreover the zero section has closed image, then $\supp(f)$ is an open subset of $\ker\OO$, from which we get the inclusion $\{f\in\secring\mid \supp(f)\scj\ker\OO\}\scj \Gamma_c(\mathrm{Int}(\ker\OO),\OO)$. By Lemma~\ref{l:centralizer}, we then get that $\secringunit$ is maximal commutative if and only if $\secringunit=\Gamma_c(\mathrm{Int}(\ker \mathcal O),\mathcal O)$.  But this is trivially equivalent to $\grp^{(0)}=\mathrm{Int}(\ker \mathcal O)$.
\end{proof}

\begin{Thm}[Generalized Uniqueness Theorem]\label{gut}
 Let $\grp$ be an ample groupoid and $\mathcal O$ a $\grp$-sheaf of commutative rings. A ring homomorphism $\pi:\secring\to A$ is injective if and only if $\pi|_{C_{\secring}(\secringunit)}$ is injective.
\end{Thm}

\begin{proof}
The ``only if'' part is immediate. For the ``if'' part, we notice that since $\secring$ is a skew inverse semigroup ring by the results of Section \ref{sec:interplay}, we can use the second part of the proof of \cite[Theorem~3.4]{BGOR} to show that if $I$ is a non-zero ideal of $\secring$, then $I\cap C_{\secring}(\secringunit)\neq\{0\}$ from which the result follows.
\end{proof}

In the Hausdorff case, we can use Lemma \ref{l:centralizer} and Corollary \ref{cor:masa.Hausdorff} to obtain the following.

\begin{Cor}
Let $\grp$ be a Hausdorff ample groupoid and $\mathcal O$ a $\grp$-sheaf of integral domains such that the zero section is closed (eg. $\OO$ is a $\grp$-sheaf of fields). A ring homomorphism $\pi:\secring\to A$ is injective if and only if $\pi|_{\Gamma_c(\mathrm{Int}(\ker\OO),\OO)}$ is injective. If in addition $\mathrm{Int}(\ker\OO)=\grp^{(0)}$ (eg. $\grp$ is effective), $\pi$ is injective if and only if $\pi|_{\secringunit}$ is injective.
\end{Cor}

\subsection{Primitivity of $\secring$}

In this section, we give some necessary conditions and some sufficient conditions for $\secring$ to be left primitive. Under certain hypotheses on the sheaf, the condition will be both necessary and sufficient. Recall that a ring is said to be \emph{left primitive} if it has a faithful simple left module.  The following theorem generalizes a result of~\cite{groupoidprimitive} for Steinberg algebras.

\begin{Thm}\label{t:main.primitive}
Let $\grp$ be an ample groupoid and $\mathcal O$ a $\mathscr G$-sheaf of rings.

\begin{enumerate}
    \item If $\secring$ is left primitive, then there is an orbit $\mathrm{Orb}(x)$ such that, for $f\in \secring$, $f|_{\dom\inv (\mathrm{Orb}(x))}=0$ implies $f=0$.
    \item Suppose that there is an orbit $\mathrm{Orb}(x)$ with $B_x$ left primitive and such that $f|_{\dom\inv (\mathrm{Orb}(x))}=0$ implies $f=0$ for $f\in \secring$. Then $\secring$ is left primitive.
\end{enumerate}
\end{Thm}
\begin{proof}
Let $M$ be a faithful simple left $\secring$-module. 
There exists a $\grp$-sheaf of $\OO$-modules $\mathcal M$ such that $M=\secmod$ by the Disintegration Theorem. By Theorem \ref{thm:primitive.ideal}, there exists $x\in\grp^{(0)}$ such that $\{0\}=\Ann(M)=\Ann(\Ind_x(\mathcal M_x))$. We claim that $f|_{\dom\inv (\mathrm{Orb}(x))}=0$ implies $f=0$ for $f\in \secring$.  Indeed, $f$ annihilates $\LLL_x$ since if $\dom(\gamma)=x$ and $a\in \mathcal O_{\ran(\gamma)}$, then $f\cdot a1_\gamma = \sum_{\dom(\beta)=\ran(\gamma)} f(\beta)\alpha_{\beta}(a)1_{\beta\gamma}=0$ as $\ran(\gamma)\in \mathrm{Orb}(x)$.  Thus $f$ annihilates $M$. Since $\{0\}=\Ann(M)=\Ann(\Ind_x(\mathcal M_x))$, we conclude that $f=0$. 

For (2), let $M$ be a faithful simple left $B_x$-module and consider $N=\Ind_x(M)$ the induced module. That $N$ is a simple left $\secring$-module follows from Theorem \ref{thm:simple.module}.

We now prove that $N$ is faithful. For that, we fix $f\in\secring\setminus\{0\}$. By assumption, there exists $\gamma\in\supp(f)$ such that $\dom(\gamma)\in \mathrm{Orb}(x)$. We choose elements $\eta_y$ for $y\in \mathrm{Orb}(x)$ as per \eqref{basis}. Consider the following element of $B_x$:
\[a=\sum_{\zeta:\dom(\gamma)\to\ran(\gamma)}\alpha_{\eta_{\ran(\gamma)}\inv}(f(\zeta))\eta_{\ran(\gamma)}\inv\zeta\eta_{\dom(\gamma)},\]
and notice that $a\neq 0$, since the coefficient of $\eta_{\ran(\gamma)}\inv\gamma\eta_{\dom(\gamma)}$ is $\alpha_{\eta_{\ran(\gamma)}\inv}(f(\gamma))$, which is non-zero. Since $M$ is faithful, there exists $m\in M$ such that $a\cdot m\neq 0$. By \eqref{eq:Gamma.acts.induced.module}, the component of $f\cdot 1_{\eta_{\dom(\gamma)}}\otimes m$ corresponding to the coordinate given by $1_{\eta_{\ran(\gamma)}}$ is
\[
    \sum_{\zeta:\dom(\gamma)\to \ran(\gamma)}1_{\eta_{\ran(\gamma)}}\otimes\alpha_{\eta_{\ran(\gamma)}\inv}(f(\zeta))\eta_{\ran(\gamma)}\inv\zeta\eta_{\dom(\gamma)} \cdot m = 1_{\eta_{\ran(\gamma)}}\otimes a\cdot m\neq 0.
\]
Hence $f\cdot 1_{\eta_{\dom(\gamma)}}\otimes m\neq 0$, which proves the faithfulness of $N$.
\end{proof}

The first condition in Theorem~\ref{t:main.primitive} implies that if $\secring$ is left primitive, then $\grp$ has a dense orbit.

\begin{Cor}\label{c:dense.orbit}
Let $\mathscr G$ be an ample groupoid and $\mathcal O$ a $\mathscr G$-sheaf of rings. If $\secring$ is left primitive, then there exists $x\in\grp^{(0)}$ such that $\overline{\mathrm{Orb}(x)}=\grp^{(0)}$.
\end{Cor}
\begin{proof}
Let $\mathrm{Orb}(x)$ be as in Theorem~\ref{t:main.primitive}(1). If $\emptyset\neq U$ is a compact open subset of $\grp^{(0)}$, then $\chi_U\neq 0$ and so there exists $y\in \mathrm{Orb}(x)$ with $\chi_U(y)\neq 0$, that is, $y\in U$.  We conclude that $\mathrm{Orb}(x)$ is dense.
\end{proof}

\begin{Rmk}
Notice that, in the case of the constant sheaf (as in Example~\ref{ex:constant.sheaf}) the above result generalizes \cite[Proposition~4.9]{groupoidprimitive} to allow for usual groupoid algebras over possibly non-commutative rings (as opposed to fields).
\end{Rmk}

\begin{Cor}
Let $\mathscr G$ be a Hausorff ample groupoid and $\mathcal O$ a $\mathscr G$-sheaf of rings such that the zero section has closed image. If there exists $x\in\grp^{(0)}$ such that $\overline{\mathrm{Orb}(x)}=\grp^{(0)}$ and $B_x$ is left primitive, then $\secring$ is left primitive.
\end{Cor}

\begin{proof}
If $0\neq f\in \secring$ then, since $f$ is continuous and the image of the zero section is closed, $\supp f$ is open, whence $\dom(\supp f)$ is open.  By density of $\mathrm{Orb}(x)$, we conclude there exists $\gamma\in \dom\inv(\mathrm{Orb}(x))\cap \supp(f)$.  The result now follows from Theorem~\ref{t:main.primitive}(2).
\end{proof}

Next, we give necessary and sufficient conditions for $\secring$ to be left primitive in the case that $\mathcal O$ is a sheaf of fields and $\secringunit$ is a maximal commutative subring.

\begin{Thm}\label{t:max.abel.case}
Suppose that $\mathcal O$ is a $\grp$-sheaf of fields and $\secringunit$ is a maximal commutative subring.  Then $\secring$ is left primitive if and only if $\grp^{(0)}$ has a dense orbit.
\end{Thm}
\begin{proof}
Corollary~\ref{c:dense.orbit} shows that the condition is necessary.  For the converse, let $x$ have a dense orbit, $B_x$ be as usual and let $M$ be any simple $B_x$-module (since $B_x$ is unital, these exist).  We verify that $\Ind_x(M)$ is a faithful left $\secring$-module.   It is simple by Theorem~\ref{thm:simple.module}.  

Let $I$ be the annihilator of $\Ind_x(M)$.  Recall from~\cite{BenDan} that $\secring$ is a skew inverse semigroup ring $\secringunit\rtimes S$ for an appropriate inverse semigroup $S$.  By~\cite[Theorem~3.4]{BGOR}, since $\secringunit$ is maximal commutative, if $I\neq 0$, then there exists $0\neq s\in I\cap \secringunit$.  Since $s$ is continuous and the orbit of $x$ is dense, $0\neq s(y)$ for some $y\in \mathrm{Orb}(x)$.   Let $\gamma\colon x\to y$ and let $U$ be a compact open bisection containing $\gamma$.  Then $\chi_{U\inv}\ast s\ast \chi_U\in I\cap \secringunit$ and $\chi_{U\inv}\ast s\ast \chi_U(x) =  \alpha_{\gamma\inv}(s(y))\neq 0$.  Thus without loss of generality, we may assume that $s(x)\neq 0$.  Since $s(x)$ is a unit of $B_x$, $s(x)m\neq 0$ for all $m\in M\setminus \{0\}$.  So if $0\neq m\in M$, then $s\cdot x\otimes m = s(x)x\otimes m =x\alpha_x(s(x))\otimes m= xs(x)\otimes m= x\otimes s(x)m\neq 0$, contradicting that $s\in I$.  Thus $I=0$.  This completes the proof.
\end{proof}

\begin{Cor}\label{azul}
Let $\grp$ be a Hausdorff ample groupoid and $\mathcal O$ a $\grp$-sheaf of fields.  Assume that $\mathrm{Int}(\ker \mathcal O)=\grp^{(0)}$, e.g., if $\grp$ is effective.  Then, $\secring$ is left primitive if and only if $\grp^{(0)}$ has a dense orbit.  
\end{Cor}

\subsection{Semiprimitivity}
Recall that a ring is \emph{semiprimitive} if its Jacobson radical is zero.  Equivalently, it is semiprimitive if it has a faithful semisimple module.  There are many open questions about the semiprimitivity of group rings and skew group rings, in particular, it is still unknown if a group algebra over a field of characteristic $0$ is necessarily semiprimitive. So we shall endeavor to understand things modulo this situation.  The results of this section generalize the semiprimitivity results for Steinberg algebras from~\cite{groupoidprimitive}.

Recall that a subset $D\scj \grp^{(0)}$ is \emph{invariant} if for all $\gamma\in\grp$, $\dom(\gamma)\in D$ if and only if $\ran(\gamma)\in D$. If $X$ is an invariant subset of $\grp^{(0)}$, then $\grp|_X$ will denote the restriction of $\grp$ to $X$. 

\begin{Thm}\label{t:main.semiprimitive}
Let $\grp$ be an ample groupoid and $\mathcal O$ a $\mathscr G$-sheaf of rings.  
Suppose that there is an invariant subset $X$ with $B_x$ semiprimitive for some $x$ in each orbit of $X$ and such that $f|_{\grp|_X}=0$ implies $f=0$ for $f\in \secring$. Then $\secring$ is semiprimitive.
\end{Thm}
\begin{proof}
Let $T$ be a set of orbit representatives of $X$ with $B_x$ semiprimitive for all $x\in T$. Let $M_x$ be a faithful left semisimple $B_x$-module and consider $N=\bigoplus_{x\in T}\Ind_x(M_x)$. Then $N$ is a semisimple left $\secring$-module by Theorem \ref{thm:simple.module}.

We now prove that $N$ is faithful. For that, we fix $f\in\secring\setminus\{0\}$. By assumption, there exists $\gamma\in\grp|_X$ such that $f(\gamma)\neq 0$.  Say that $\dom(\gamma)\in \mathrm{Orb}(x)$ with $x\in T$. Again, we use the notation of \eqref{basis}. Consider the following element of $B_x$:
\[a=\sum_{\zeta:\dom(\gamma)\to\ran(\gamma)}\alpha_{\eta_{\ran(\gamma)}\inv}(f(\zeta))\eta_{\ran(\gamma)}\inv\zeta\eta_{\dom(\gamma)},\]
and notice that $a\neq 0$, since the coefficient of $\eta_{\ran(\gamma)}\inv\gamma\eta_{\dom(\gamma)}$ is $\alpha_{\eta_{\ran(\gamma)}\inv}(f(\gamma))$, which is non-zero. Since $M_x$ is faithful, there exists $m\in M_x$ such that $a\cdot m\neq 0$. By \eqref{eq:Gamma.acts.induced.module}, the component of $f\cdot 1_{\eta_{\dom(\gamma)}}\otimes m$ corresponding to the coordinate given by $1_{\eta_{\ran(\gamma)}}$ is
\[
    \sum_{\zeta:\dom(\gamma)\to \ran(\gamma)}1_{\eta_{\ran(\gamma)}}\otimes\alpha_{\eta_{\ran(\gamma)}\inv}(f(\zeta))\eta_{\ran(\gamma)}\inv\zeta\eta_{\dom(\gamma)} \cdot m = 1_{\eta_{\ran(\gamma)}}\otimes a\cdot m\neq 0.
\]
Hence $f\cdot 1_{\eta_{\dom(\gamma)}}\otimes m\neq 0$, which proves the faithfulness of $N$.
\end{proof}
 
\begin{Cor}
Let $\mathscr G$ be a Hausorff ample groupoid and $\mathcal O$ a $\mathscr G$-sheaf of rings such that the zero section has closed image (e.g., a sheaf of fields). If there exists a dense invariant subset $X\subseteq \grp^{(0)}$ with  $B_x$ semiprimitive for some $x$ in each orbit, then $\secring$ is semiprimitive.
\end{Cor}
\begin{proof}
If $0\neq f\in \secring$ then, since $f$ is continuous and the image of the zero section is closed, $\supp f$ is open, whence $\dom(\supp f)$ is open.  By density of $X$, we conclude there exists $\gamma\in\grp|_X\cap \supp(f)$.  The result now follows from Theorem~\ref{t:main.semiprimitive}.
\end{proof}

\begin{Thm}\label{tsemiprimitivity}
Suppose that $\mathcal O$ is a $\grp$-sheaf of fields and $\secringunit$ is a maximal commutative subring.  Then $\secring$ is semiprimitive.
\end{Thm}
\begin{proof}
For $x\in \grp^{(0)}$, let $B_x$ be the skew group ring, as usual, and let $M_x$ be any simple $B_x$-module (since $B_x$ is unital, these exist).  We verify that $M=\bigoplus_{x\in \grp^{(0)}}\Ind_X(M_x)$ is a faithful semisimple left $\secring$-module.   It is semisimple by Theorem~\ref{thm:simple.module}.  

Let $I$ be the annihilator of $M$.  Recall from~\cite{BenDan} that $\secring$ is a skew inverse semigroup ring $\secringunit\rtimes S$ for a suitably chosen inverse semigroup $S$.  By~\cite[Theorem~3.4]{BGOR}, since $\secringunit$ is maximal commutative, if $I\neq 0$, then there exists $0\neq s\in I\cap \secringunit$.  Suppose that $s(x)\neq 0$ with $x\in \grp^{(0)}$.  Since $s(x)$ is a unit of $B_x$, $s(x)m\neq 0$ for all $m\in M_x\setminus \{0\}$.  So if $0\neq m\in M_x$, then $s\cdot x\otimes m = s(x)x\otimes m =x\alpha_x(s(x))\otimes m= xs(x)\otimes m= x\otimes s(x)m\neq 0$, contradicting that $s\in I$.  Thus $I=0$.  This completes the proof.
\end{proof}

\begin{Cor}\label{esqui}
Let $\grp$ be a Hausdorff ample groupoid and $\mathcal O$ a $\grp$-sheaf of fields.  Assume that $\mathrm{Int}(\ker \mathcal O)=\grp^{(0)}$, e.g., if $\grp$ is effective.  Then $\secring$ is semiprimitive.  
\end{Cor}

\subsection{Simplicity}

Under the hypothesis that we have a sheaf of fields, we give necessary and sufficient conditions for the ring $\secring$ to be simple.

 We say that $\grp$ is \emph{minimal} if the only open invariant subsets of $\grp^{(0)}$ are $\{0\}$ and $\grp^{(0)}$. For the action of $\grp^a$ on $\secringunit$, given in Section~\ref{sec:interplay}, we say that an ideal $I$ of $\secringunit$ is \emph{$\grp^a$-invariant} if for all $s\in\grp^a$, $\til \alpha_s(D_s^*\cap I)\scj I$. Finally, $\secringunit$ is \emph{$\grp^a$-simple} if the only $\grp^a$-invariant ideals of $\secringunit$ are $\{0\}$ and $\secringunit$. 

We aim to show that minimality is equivalent to $\grp^a$-simplicity for a large class of $\mathscr G$-sheaves of commutative rings. To see that some hypothesis on $\mathcal O$ is needed, let $\grp$ be any ample groupoid and $R$ a commutative ring that is not a field.  Let $\mathcal O$ be the constant sheaf associated with $R$.  Let $I$ be a nonzero proper ideal in $R$.  Then the set of functions in  $\secringunit$ that take values in $I$ is a nonzero, proper $\grp^a$-invariant ideal that contains no characteristic function $\chi_V$ with $V\neq \emptyset$. So minimality does not imply $\grp^a$-simplicity in this case.

In what follows, for $\grp$ an ample groupoid, $\mathcal O$ a $\grp$-sheaf of rings and $U\scj\grp^{(0)}$ open, we define the ideal $I_U=\{f\in\secringunit\mid\supp f\scj U\}$.

\begin{Lemma}\label{l:ideal.open.set}
Let $\grp$ be an ample groupoid and $\mathcal O$ a $\grp$-sheaf of rings.
\begin{enumerate}
    \item\label{i:open1} If $U\scj \grp^{(0)}$ is open, then $I_U$ is generated by $\chi_V$ with $V\scj U$ compact-open.
    \item\label{i:open2} If $I$ is an ideal of $\secringunit$ and $U=\bigcup_{\chi_V\in I}V$, then $I_U$ is generated by $\chi_V$ such that $\chi_V\in I$. In particular, $I_U\scj I$.
\end{enumerate}
\end{Lemma}
\begin{proof}
\eqref{i:open1} Let $f\in I_U$. Since $\supp(f)$ is compact and $\grp^{(0)}$ has a basis of compact opens, we can find a compact open set $V$ with $\supp(f)\subseteq V\subseteq U$ (choose a compact open neighborhood of each $x\in \supp(f)$ contained in $U$ and take the union of a finite subcover).  Then $f=f\chi_V$ and so $I_U$ is indeed generated by the $\chi_V$ with $V\subseteq U$ compact open.

\eqref{i:open2} Notice that if $V,W$ are such that $\chi_V,\chi_W\in I$, then $\chi_{V\cup W}$, $\chi_{V\setminus W}$ and $\chi_{V\cap W}$ are all in $I$. Also, if $\chi_V\in I$ and $W\subseteq V$ is compact open, then $\chi_W = \chi_W\chi_V\in I$.  It now follows that if $W\subseteq U$ is compact-open, then $W\subseteq W_1\cup\cdots\cup W_n$ with $W_i$ such $\chi_{W_i}\in I$ for all $i=1,\ldots,n$, so that $\chi_W\in I$. It follows that $I_U\subseteq I$ by \eqref{i:open1}.
\end{proof}

\begin{Lemma}\label{l:correspondece.open.ideal}
Let $\grp$ be an ample groupoid and $\mathcal O$ a $\grp$-sheaf of indecomposable rings. Then there is a one-to-one correspondence between open subsets of $\grp^{(0)}$ and ideals of $\secringunit$ generated by idempotents, where if $U\scj\grp^{(0)}$ is open, we define the ideal $I_U=\{f\in\secringunit\mid\supp f\scj U\}$, and if $I\scj\secringunit$ is an ideal generated by idempotents, we define $U_I=\bigcup_{\chi_V\in I} V$.  In particular, if $\mathcal O$ is a $\grp$-sheaf of fields, then there is a bijection between open subsets of $\grp^{(0)}$ and ideals of $\secringunit$.
\end{Lemma}
\begin{proof}
If $U\scj \grp^{(0)}$ is open, then, by Lemma~\ref{l:ideal.open.set}\eqref{i:open1}, $I_U$ is generated by the $\chi_V$ with $V\subseteq U$ compact open. Thus $I_U$ is generated by idempotents and clearly $U=\bigcup_{\chi_V\in I_U} V$ since $\grp^{(0)}$ has a basis of compact opens.  

Conversely, suppose that $I$ is an ideal generated by idempotents. Since $\mathcal O$ is a sheaf of indecomposable rings, the idempotents of $\secringunit$ are the characteristic functions $\chi_V$ with $V\subseteq \grp^{(0)}$ compact open. It now follows that $I\subseteq I_U$ where $U=\bigcup_{\chi_V\in I} V$. By Lemma~\ref{l:ideal.open.set}\eqref{i:open2}, we have the reverse inclusion $I_U\scj I$.

The final statement follows because if $\mathcal O$ is a $\grp$-sheaf of fields, then $\secringunit$ is von Neumann regular by Proposition~\ref{p:vnr} and hence every ideal is generated by idempotents (since if $aba=a$, then $ba$ and $a$ generate the same ideal and $ba$ is idempotent). 
\end{proof}

\begin{Prop}\label{p:minimal.cond}
Let $\grp$ be an ample groupoid and $\mathcal O$ a $\grp$-sheaf of rings such that each nonzero ideal $I$ of $\secringunit$ contains a characteristic function $\chi_V$ with $\emptyset\neq V\subseteq \grp^{(0)}$ compact open.   Then the groupoid $\grp$ is minimal if and only if $\secringunit$ is $\grp^a$-simple.
\end{Prop}
\begin{proof}
First, suppose that $\grp$ is minimal.
Let $I$ be a nonzero $\grp^a$-invariant ideal and let $U=\bigcup_{\chi_V\in I}V$.  Then by Lemma~\ref{l:ideal.open.set}\eqref{i:open2}, we have that $I_U$ is generated by those $\chi_V$ in $I$ and hence $I_U\subseteq I$.  Also, by hypothesis, $I_U\neq 0$ and so $U\neq \emptyset$.  We claim that $U$ is $\grp^a$-invariant.  Indeed, if $x\in U$ and $\gamma\colon x\to y$, then by definition there is $\chi_V\in I$ with $x\in V$.  Choose $W\in \grp^a$ with $\gamma\in W$.  Then $\widetilde\alpha_W(\chi_V)=\chi_{WVW^{-1}}=\chi_W\chi_V\chi_{W^{-1}}$ is in $I$ and $y\in WVW^{-1}$. Thus $y\in U$ and so $U$ is invariant.  We conclude that $U=\grp^{(0)}$ and so $\secringunit=I_U\subseteq I$, as required.

Conversely, suppose that $\secringunit$ is $\grp^a$-minimal.  Let $U$ be a nonempty open invariant subset of $\grp^{(0)}$.  Then we claim $I_U$ is $\grp^a$-invariant.  It suffices to show that if $W$ is a compact open bisection and $V\subseteq U$ is compact open, then $\widetilde\alpha_W(\chi_V)\in I_U$, as the $\chi_V$ with $V\subseteq U$ generate $I_U$ by Lemma~\ref{l:ideal.open.set}\eqref{i:open1}.  But, as before, $\widetilde\alpha_W(\chi_V)=\chi_{WVW^{-1}}$ and $WVW^{-1}\subseteq U$ by invariance.  Thus $\widetilde\alpha_W(\chi_V)\in I_U$.  We conclude that $I_U=\secringunit$, as $I_U\neq 0$, and so $U=\grp^{(0)}$.  
\end{proof}

We now specialize the above results. Recall that a commutative ring $R$ with identity is \emph{local} if the noninvertible elements of $R$ form an ideal $\mathfrak m$, which is necessarily the unique maximal ideal of $R$.  Every field is local and every local ring is indecomposable. 

A sheaf of fields is a sheaf of local rings.  Another example is the following: if $X$ is a locally compact, totally disconnected space and $\mathcal O$ is the sheaf of germs of continuous complex (or real) valued functions on $X$, then $\mathcal O$ is a sheaf of local rings.  The maximal ideal $\mathfrak m_x$ of $\mathcal O_x$ consists of those germs of functions that vanish at $x$.  The global sections correspond exactly to continuous functions $f$ on $X$; the corresponding section sends $x$ to the germ of $f$ at $x$.  Thus the section corresponding to $f$ vanishes at $x\in X$ if and only if $f(x)=0$ and so only the zero section is noninvertible at every point of $X$.  This motivates the following extension of~\cite[Proposition 5.4]{BGOR}

\begin{Cor}\label{c:minimal}
Let $\grp$ be an ample groupoid and $\mathcal O$  a $\grp$-sheaf of rings such  that every nonzero section $s\in \secringunit$ is invertible at some $x\in \grp^{(0)}$.   Then, the groupoid $\grp$ is minimal if and only if $\secringunit$ is $\grp^a$-simple. In particular, the statement of the corollary is true if $\mathcal O$ is a sheaf of fields.
\end{Cor}
\begin{proof}

Let $I$ be a nonzero ideal of $\secringunit$ and $0\neq s\in I$.  Then by hypothesis, there is $x\in \grp^{(0)}$ such that $s(x)\in \mathcal O_x^\times$ and hence, since $\mathcal O^\times$ is open by Lemma~\ref{units:open}, we conclude there is a compact open neighborhood $W$ of $x$ such that $s(W)\subseteq \mathcal O^\times$.  Since inversion is continuous on $\mathcal O^\times$ by Lemma~\ref{units:open}, there is a section $t$ defined on $W$ with $t(y)=s(y)^{-1}$ for all $y\in W$.  Then $t\in \secringunit$, $\chi_W=st\in I$ and $W\neq\emptyset$.  The result now follows from Proposition~\ref{p:minimal.cond}.

The final statement follows because if $\mathcal O$ is a sheaf of fields, then $s(x)\neq 0$ implies $s(x)$ is invertible.

\end{proof}

We remark that if $\mathcal O$ is a $\grp$-sheaf of commutative local rings, then the hypotheses of Proposition~\ref{p:minimal.cond} are equivalent to the hypotheses of Corollary~\ref{c:minimal}.  Indeed, the proof of Corollary~\ref{c:minimal} provides one implication.  On the other hand, for a $\grp$-sheaf of local rings, the sections $s\in \secringunit$ that are nowhere invertible form an ideal that contains no nonzero characteristic function.  Thus if the hypotheses of  Proposition~\ref{p:minimal.cond} apply, then there are no nonzero nowhere invertible sections.


\begin{Thm}\label{simplelife}
Let $\grp$ be an ample groupoid and $\mathcal O$ a $\grp$-sheaf of fields. The following are equivalent:
\begin{enumerate}
    \item\label{t.item:simple1} $\secring$ is simple;
    \item\label{t.item:simple2} $\secringunit$ is $\grp^a$-simple and a maximal commutative subring of $\secring$;
    \item\label{t.item:simple3} $\grp$ is minimal, $\mathrm{Int}(\ker \mathcal O)=\grp^{(0)}$ and $\Gamma_c(\mathrm{Int}(\ker\OO),\OO)=\{f\in\secring\mid \supp(f)\scj\ker\OO\}$.
\end{enumerate}
If moreover, $\grp$ is Hausdorff, the above items are also equivalent to:
\begin{enumerate}  \setcounter{enumi}{3}
    \item\label{t.item:simple4} $\grp$ is minimal and $\mathrm{Int}(\ker \mathcal O)=\grp^{(0)}$.
\end{enumerate}
\end{Thm}

\begin{proof}
The equivalence between \eqref{t.item:simple1} and \eqref{t.item:simple2} is given by \cite[Theorem 3.7]{BGOR}. The equivalence between \eqref{t.item:simple2} and \eqref{t.item:simple3} follows from Propositions \ref{p:neweffective} and \ref{c:minimal}. In the Hausdorff case, the equivalence between \eqref{t.item:simple2} and \eqref{t.item:simple4} is due to Corollary \ref{cor:masa.Hausdorff} and Proposition \ref{c:minimal}.
\end{proof}




\section{Applications to topological dynamics}\label{dynamics}

 Motivated by our results on primitivity and semiprimitivity, described in Corollary~\ref{azul} and Theorem~\ref{tsemiprimitivity}, in this section, we characterize the existence of a dense orbit in an action of an inverse semigroup on a topological space in terms of primitivity of the associated algebra. We will also show that the algebras associated with effective actions are always semiprimitive.



An important class of actions is formed by topologically free\footnote{The terminology ``effective'' is used in \cite{BC} but the more usual meaning of effective action is faithful and, indeed, the origin of the term ``effective groupoid'' is that $\grp$ is effective if and only if the action of the inverse semigroup of open bisections on $\grp^{(0)}$ is faithful.}  actions~\cite{BC, BG, DG, GOR, groupoidprimitive}. We recall the definition, in the context of inverse semigroup actions, below. 

\begin{Def}\label{def:topfree}[\!\!{\cite[Definition~4.1]{ExelPardo}}]
Let $\theta=\left(\{X_s\}_{s\in S},\{\theta_s\}_{s\in S}\right)$ be an action of an inverse semigroup $S$ on a locally compact Hausdorff space $X$ by partial homeomorphisms. We say that $\theta$ is 
\emph{topologically free} if, and only if, the interior of the set $\left\lbrace x \in X_{s^*} \mid  \theta_{s}(x)=x \right\rbrace$ is equal to 
 \begin{displaymath}
 \{ x \in X_{s^*}  \mid \text{ there is } \ e \in E(S) \text{ such that } e\leq s \text{ and } x \in X_e \},
\end{displaymath}
for all $s \in S$.
\end{Def}

We will see below that the groupoid of germs associated with an inverse semigroup action is the key object connecting properties of the action with properties of the associated algebras. We refer the reader to the second paragraph below Theorem~\ref{siri}, or to \cite[Section~3]{BC}, for the definition of the groupoid of germs. 

The following is proved in \cite[Proposition~4.7]{ExelPardo}, \cite[Proposition 5.6]{groupoidprimitive}  and \cite[Proposition~7.3]{BC}.

\begin{Prop}\label{cinza} Let $\theta$ be an action of the inverse semigroup $S$ on the topological space $X$. Then, $\theta$ is topologically free if, and only if, the corresponding groupoid of germs $S\ltimes X$ is effective.
\end{Prop}

\begin{Prop}\label{orbit} Given an action $\theta$ of the inverse semigroup $S$ on the topological space $X$, the groupoid of germs $S\ltimes X$ has a dense orbit if and only if $\theta$ has a dense orbit.
\end{Prop}
\begin{proof}
 This is immediate, since the unit space of $S\ltimes X$ is identified with $X$, and the groupoid orbit of a point $x\in X$ is equal to the orbit of $x$ under the action of $S$, as shown in~\cite[Proposition~5.4]{groupoidprimitive}.
 %
\end{proof}

\begin{Rmk} If $X$ is locally compact, Hausdorff and second countable, having a dense orbit is equivalent to topological transitivity of the action, see for example~\cite[Lemma~3.4]{groupoidprime} (or for a special case \cite[Theorem~4.1]{BCS}).
\end{Rmk}

Next, we recall how to associate an algebra to a topological inverse semigroup action.
Let $R$ be a unital and commutative ring, and $\theta=(\{\theta_s \}_{s\in S}, \{X_s\}_{s\in S})$ be an action of an inverse semigroup $S$ on a locally compact, Hausdorff, zero-dimensional space $X$ (see \cite{BC, BG} for details on actions of inverse semigroups). There is a corresponding dual action (see \cite{BC}) $\alpha=(\{\alpha_s\}_{s\in S}, \{D_s\}_{s\in S})$ of the semigroup $S$ on the $R$-algebra $\Lc$ of all locally constant, compactly supported, $R$-valued functions on $X$ where, for each $s \in S$, $\alpha_s$ is the isomorphism from $D_{s^*}= \mathcal{L}_c(X_{s^*})$ 
onto $D_s= \mathcal{L}_c(X_s)$ given by 
\begin{displaymath}
	\alpha_s(f)(x)= \left\lbrace \begin{array}{ccr}
    f \circ \theta_{s^*}(x) & & \mbox{if} \,\, x \in X_s \\
    0                       & & \mbox{if} \,\, x \notin X_s 
    \end{array}\right..
\end{displaymath}


By \cite[Theorem~5.10]{BC}, the skew inverse semigroup ring, $\mathcal{L}_c(X)\rtimes S$, associated to the action $\alpha$ above, is isomorphic to the usual groupoid algebra, denoted by $R (S\ltimes X)$, over the groupoid of germs $ S\ltimes X$. Recall from Example~\ref{ex:constant.sheaf} that $R (S\ltimes X)$ is isomorphic to  $\Gamma_c(S\ltimes X,\Delta(R))$, where $\Delta(R)$ is the constant sheaf over $S\ltimes X$. 
More generally, we can consider a $S\ltimes X$-sheaf of rings $\OO$, which includes the example above as well as the $S\ltimes \wh{A}$-sheaf $\OO_A$ coming from a spectral action of an inverse semigroup $S$ on a unital ring $A$ as in Theorem \ref{t:sheaf.Pierce}. 

We will use the results developed in Section \ref{s:applications} to relate algebraic properties of the ring $\Gamma_c(S\ltimes X,\OO)$ and topological properties of the action of $S$ on $X$. 
When the groupoid of germs $S\ltimes X$ is not Hausdorff, our results depend on the maximal commutativity of $ \Gamma_c(\mathscr (S\ltimes X)^{(0)},\mathcal O)$. As we showed in Corollary~\ref{cor:masa.Hausdorff}, when $S\ltimes X$ is Hausdorff this condition is simplified and, in particular, if $\theta$ is topologically free then $ \Gamma_c(\mathscr (S\ltimes X)^{(0)},\mathcal O)$ is maximal commutative. For sufficient conditions for the groupoid of germs to be Hausdorff, see  \cite[Proposition~3.20]{BC} and \cite[Theorem~3.15]{ExelPardo} for example.



\begin{Prop}\label{oneway}
Let $\theta=(\{\theta_s \}_{s\in S}, \{X_s\}_{s\in S})$ be a action of an inverse semigroup $S$ on a locally compact, Hausdorff, zero-dimensional space $X$, and let $\OO$ be a $S\ltimes X$-sheaf of rings. If $\Gamma_c(S\ltimes X,\OO)$ is left primitive then $\theta$ has a dense orbit.
\end{Prop}
\begin{proof}
By Corollary~\ref{c:dense.orbit} the groupoid of germs $S\ltimes X$ has a dense orbit. The result now follows from Proposition~\ref{orbit}.
\end{proof}

In the case of a sheaf of fields and $ \Gamma_c(\mathscr (S\ltimes X)^{(0)},\mathcal O)$ maximal commutative we have a converse of the above result.

\begin{Thm} Let $\theta=(\{\theta_s \}_{s\in S}, \{X_s\}_{s\in S})$ be an action of an inverse semigroup $S$ on a locally compact, Hausdorff, zero-dimensional space $X$, and let $\OO$ be a $S\ltimes X$-sheaf of fields. Suppose that $ \Gamma_c(\mathscr (S\ltimes X)^{(0)},\mathcal O)$ is maximal commutative.  Then $\theta$ has a dense orbit if, and only if, $\Gamma_c(S\ltimes X,\OO)$ is left primitive.

\end{Thm}
\begin{proof}

This follows from the Theorem~\ref{t:max.abel.case} and Proposition~\ref{orbit}.

\end{proof}

If $S\ltimes X$ is Hausdorff and the action is topologically free, then we do not need the assumption that $ \Gamma_c(\mathscr (S\ltimes X)^{(0)},\mathcal O)$ is maximal commutative, as we see below.

\begin{Cor}\label{primodenso} Let $\theta=(\{\theta_s \}_{s\in S}, \{X_s\}_{s\in S})$ be a topologically free action of an inverse semigroup $S$ on a locally compact, Hausdorff, zero-dimensional space $X$, and let $\OO$ be a $S\ltimes X$-sheaf of fields. If $S\ltimes X$ is Hausdorff, then $\theta$ has a dense orbit if, and only if, $\Gamma_c(S\ltimes X,\OO)$ is left primitive.

\end{Cor}
\begin{proof}

This follows from Corollary~\ref{azul} and Proposition~\ref{orbit}.

\end{proof}


\begin{Rmk} If $R$ is a field then, taking $\OO = \Delta(R)$ in the corollary above, we obtain that an action $\theta$ on $X$, whose groupoid of germs is Hausdorff, has a dense orbit if, and only if, $\mathcal{L}_c(X)\rtimes S$ is left primitive. 
\end{Rmk}

\begin{Rmk}
Since $\Gamma_c(S\ltimes X,\Delta(R))$ coincides with the groupoid algebra of \cite{Steinbergalgebra}, the above corollary can also be obtained using Theorem~4.10 in \cite{groupoidprimitive}. 
\end{Rmk}

Next, we observe that for a sheaf of fields and $ \Gamma_c(\mathscr (S\ltimes X)^{(0)},\mathcal O)$ maximal commutative, the algebra $\Gamma_c(S\ltimes X,\OO)$ is always semiprimitive.


\begin{Prop}\label{semiprop1}
Let $\theta=(\{\theta_s \}_{s\in S}, \{X_s\}_{s\in S})$ be an action of an inverse semigroup $S$ on a locally compact, Hausdorff, zero-dimensional space $X$, and let $\OO$ be a $S\ltimes X$-sheaf of fields. If $ \Gamma_c(\mathscr (S\ltimes X)^{(0)},\mathcal O)$ is maximal commutative, then $\Gamma_c(S\ltimes X,\OO)$ is semiprimitive.
\end{Prop}
\begin{proof}
This follows Theorem~\ref{tsemiprimitivity}.
\end{proof}

As before, if $S\ltimes X$ is Hausdorff and the action is topologically free, then we do not need the assumption that $ \Gamma_c(\mathscr (S\ltimes X)^{(0)},\mathcal O)$ is maximal commutative.

\begin{Prop}\label{semiprop2}
Let $\theta=(\{\theta_s \}_{s\in S}, \{X_s\}_{s\in S})$ be a topologically free action of an inverse semigroup $S$ on a locally compact, Hausdorff, zero-dimensional space $X$, and let $\OO$ be a $S\ltimes X$-sheaf of fields. If $S\ltimes X$ is Hausdorff, Then $\Gamma_c(S\ltimes X,\OO)$ is semiprimitive.
\end{Prop}
\begin{proof}
This follows directly from Proposition~\ref{cinza} and Corollary~\ref{esqui}.
\end{proof}

Finally, we show that Theorem~\ref{simplelife} can be used to describe the minimality of a topologically free topological action in terms of the simplicity of the associated algebras. In particular, the theorem below should be compared with \cite[Corollaries~4.19 and 4.20]{BGOR}.

\begin{Thm}\label{simpleaction} Let $\theta=(\{\theta_s \}_{s\in S}, \{X_s\}_{s\in S})$ be a topologically free action of an inverse semigroup $S$ on a locally compact, Hausdorff, zero-dimensional space $X$, and let $\OO$ be a $S\ltimes X$-sheaf of fields. If $S\ltimes X$ is Hausdorff, then $\theta$ is minimal if, and only if, $\Gamma_c(S\ltimes X,\OO)$ is simple.
\end{Thm}
\begin{proof}
By Proposition~\ref{cinza}, the groupoid of germs $S\ltimes X$ is effective. Hence $\mathrm{Int}(\ker \mathcal O)=(S\ltimes X)^{(0)}$ and the result follows from Theorem~\ref{simplelife}.
\end{proof}

\section{Complex groupoid algebras}\label{complex}
Let $\grp$ be an ample groupoid. We let $C_c(\grp)$ be the usual ring used to build $C^*(\grp)$, see \cite{ConnesNH, PatersonBook}. Since $\grp$ is ample, $C_c(\grp)$ is the span of functions $f:\grp\to\cn$ for which there exists a compact-open bisection $U$ such that $\supp(f)\scj U$ and the restriction $f|_U$ is continuous, where $\cn$ has the usual topology. Our goal is to prove that $C_c(\grp)$, with the convolution product, is a groupoid ring with coefficients in a $\grp$-sheaf. In the Hausdorff case, we explicitly build a spectral action (recall the definition of a spectral action from the paragraph above Section~\ref{sec:interplay}) and show that $C_c(\grp)$ is isomorphic to the skew inverse semigroup ring of this action, so that we can apply Theorem~\ref{t:sheaf.Pierce}. In the general setting, we build a $\grp$-sheaf of rings. 

\subsection{The Hausdorff case}

Let $A:=C_c(\grp^{(0)})$ be the ring of complex valued continuous functions with compact support 
and $\mathscr G^a$ be the inverse semigroup of compact-open bisections of $\grp$ (see Section~\ref{sec:interplay}). For each $U\in \mathscr G^a$, let $D_U=\{f\in C_c(\grp^{(0)}):\supp(f)\scj\ran(U)\}$. Define a spectral action of $\mathscr G^a$ on $A$ by $\alpha_U:D_{U^{-1}}\to D_U$, where $\alpha_U(f)(\ran(\gamma))=f(\dom(\gamma))$ and $\alpha_U(f)$ vanishes outside $\ran(U)$. We then have the following.

\begin{Prop}\label{ccgpd}
With the above conditions, suppose that $\grp$ is Hausdorff. Then $C_c(\grp)$ is isomorphic as a ring to $A\rtimes \mathscr G^a$.
\end{Prop}

\begin{proof}
We first build a covariant system for $(\mathscr G^a,A,\alpha)$. The map $\theta:A\to C_c(\grp)$ is the usual inclusion of étale groupoid algebras. The map $\varphi:\mathscr G^a\to C_c(\grp)$ is given by $\varphi(U)=1_U$. A straightforward computation using the convolution product shows that $(C_c(\grp),\theta,\varphi)$ is a covariant system for $(\mathscr G^a,A,\alpha)$. By \cite[Theorem 3.5]{BenDan}, there is a ring homomorphism $\pi:A\rtimes \mathscr G^a\to C_c(\grp)$ such that $\pi(a\delta_U)=\theta(a)*\varphi(U)$ for all $U\in \mathscr G^a$ and $a\in D_{U}$.

To prove that $\pi$ is an isomorphism, it suffices to prove that $\pi$ admits an inverse. Notice that any element $f\in C_c(\grp)$ can be written as
\[f=\sum_{i=1}^n f_{U_i},\]
where $U_1,\ldots,U_n$ are pairwise disjoint compact-open bisections such that, for each $i=1,\ldots,n$, $\supp(f_{U_i})\scj U_i$ and the restriction $f_{U_i}$ to $U_i$ is continuous using that $\grp$ is Hausdorff. We want to build a map $\psi:C_c(\grp)\to A\rtimes \mathscr G^a$. For this, given $f\in C_c(\grp)$ and a decomposition as above, we set
\[\psi(f)=\sum_{i=1}^n f_{U_i}*1_{U_i\inv}\delta_{U_i}.\]
We have to prove that $\psi$ is well-defined. The general case follows from the particular case that $\supp(f)\scj U$, where $U$ is a compact-open bisection with $U=\bigcup_{i=1}^n U_i$. In this case, by the definition of $A\rtimes \mathscr G^a$, we have that
\[\sum_{i=1}^n f_{U_i}*1_{U_i\inv}\delta_{U_i}=\sum_{i=1}^n f_{U_i}*1_{U\inv}\delta_{U}=f*1_{U\inv}\delta_{U}.\]
We now prove that $\psi=\pi\inv$. First, given $f\in C_c(\grp)$ and a decomposition as above
\[\pi(\psi(f))=\sum_{i=1}^n\theta(f_{U_i}*1_{U_i\inv})*\varphi(\delta_{U_i})=\sum_{i=1}^n f_{U_i}*1_{U_i\inv}*1_{U_i}=f.\]
Observing that $\psi$ is a group homomorphism, to show that $\psi$ is a left-inverse for $\pi$, it is enough to consider an element $a\delta_U$, where $a\in D_{U}$. In this case,
\[\psi(\pi(a\delta_U))=\psi(a*1_U)=a*1_U*1_{U\inv}\delta_U=a\delta_U.\]
Hence, $\psi=\pi^{-1}$, concluding the proof.
\end{proof}

\subsection{The general case}

With the previous notation, let us build a $\grp$-sheaf $\OO$ as follows. We let $A=C_c(\grp^{(0)})$ and $\OO=A\times\grp^{(0)}/\sim$, where the equivalence relation is given by $(a,x)\sim(a',x')$ if $x=x'$ and there exists an open neighborhood $U$ of $x$ such that $a|_U=a'|_U$. The topology on $\OO$ has as basis the sets
\[D(a,U)=\{[a,x]:x\in U\},\]
where $a\in A$ and $U$ is a compact-open set of $\grp^{(0)}$. The map $p:\OO\to\grp^{(0)}$ is defined as $p([a,x])=x$, which is a local homeomorphism by the definition of the topology on $\OO$. The ring structure on each stalk is the natural one, namely $[a,x]+[b,x]:=[a+b,x]$ and $[a,x][b,x]=[ab,x]$ for $a,b\in A$ and $x\in\grp^{(0)}$. Notice that a net $\{[a_{\lambda},x_\lambda]\}_{\lambda\in\Lambda}$ converges to $[a,x]$ in $\OO$ if, and only if, for every open neighborhood $U$ of $x$, there exists $\lambda_0$ such that $x_\lambda\in U$ and $[a_\lambda,x_\lambda]=[a,x_\lambda]$ for all $\lambda \geq \lambda_0$. With this, it is straightforward to check that conditions (SR1) and (SR2) as in Section \ref{s:sheaf.coeff} are satisfied. The continuity of the unit section is immediate from the definition of the topology on $\OO$.

The map $\alpha:\grp\times_{\dom,p}\OO\to\OO$ is given by
\[\alpha_\gamma([a,\dom(\gamma)])=[a\circ\dom\circ(\ran|_U)\inv,\ran(\gamma)],\]
where $U$ is a compact bisection containing $\gamma$. Since $U$ is a bisection, the above definition does not depend on $a$. Also, since the intersection of open bisections is again an open bisection, the definition does not depend on $U$. To prove that $\alpha$ is continuous, let $\{(\gamma_\lambda,[a_\lambda,\dom(\gamma_\lambda)])\}_{\lambda\in\Lambda}$ be a net converging to $(\gamma,[a,\dom(\gamma)])$ in $\grp\times_{\dom,p}\OO$. Fix $U$ a compact-open bisection containing $\gamma$, and let $V$ be a compact-open subset of $\grp^{(0)}$ containing $\ran(\gamma)$ and such that $V\scj\ran(U)$. Using that $\ran$ is continuous, and the description of convergence in $\OO$ given above, we find $\lambda_0$ such that for all $\lambda\geq\lambda_0$, we have that $\gamma_\lambda\in U$, $[a_\lambda,\dom(\gamma_\lambda)]=[a,\dom(\gamma_\lambda)]$ and $\ran(\gamma_\lambda)\in V$. Then, for $\lambda\geq\lambda_0$, we have that
\begin{align*}
    \alpha_{\gamma_{\lambda}}([a_\lambda,\dom(\gamma_\lambda)])&=\alpha_{\gamma_{\lambda}}([a,\dom(\gamma_\lambda)])\\
    &=[a\circ\dom\circ(\ran|U)^{-1},\ran(\gamma_\lambda)]\in D(a\circ\dom\circ(\ran|U)^{-1},V).
\end{align*}
Varying $V$ as above, we obtain a neighborhood basis of $[a\circ\dom\circ(\ran|U)^{-1},\ran(\gamma)]$, from where we conclude that $\alpha$ is continuous.
Conditions (S1), (S2), and (SR4) of Section \ref{s:sheaf.coeff} are easy to check. We prove (S3). Let $\beta,\gamma\in\grp$ be such that $\dom(\beta)=\ran(\gamma)$ and let $[a,\dom(\gamma)]\in\OO$. Also, let $U$ and $V$ be open bisections containing $\beta$ and $\gamma$, respectively. We may assume, without loss of generality, that $\dom(U)=\ran(V)$. Then,
\begin{align*}
\alpha_{\beta\gamma}([a,\dom(\gamma)]) & = [a\circ \dom\circ (\ran|_{UV})\inv,\ran(\beta\gamma)] \\
 & = [a\circ \dom\circ (\ran|_{V})\inv\circ\dom\circ (\ran|_{U})\inv,\ran(\beta)] \\
 & = \alpha_{\beta}([a\circ \dom\circ (\ran|_{V})\inv,\dom(\beta)])\\
 & = \alpha_{\beta}(\alpha_{\gamma}([a,\dom(\gamma)]).
\end{align*}

\begin{Thm}\label{thm:groupoid.ring.sheaf}
Let $\grp$ be an ample groupoid and $\OO$ the $\grp$-sheaf constructed above. Then $\secring$ and $C_c(\grp)$ are isomorphic as rings.
\end{Thm}

\begin{proof}
First, let $f:\grp\to\OO$ be such that $p\circ f=\ran$ and there exists $U$ compact-open bisection of $\grp$ such that $f|_U$ is continuous and $f|_{\grp\setminus U}=0$. Define $\varphi_f:\grp\to\cn$ by $\varphi_f(\gamma)=0$, if $\gamma\notin U$, and $\varphi_f(\gamma)=a(\ran(\gamma))$, if $\gamma\in U$ and $f(\gamma)=[a,\ran(\gamma)]$. Since $f|_U$ is continuous, if $\gamma\in U$ and $f(\gamma)=[a,\ran(\gamma)]$, then there exists an open set $V$ such that $\gamma\in V\scj U$ and $f(\eta)=[a,\ran(\eta)]$ for all $\eta\in V$. This implies that $\varphi_f|_U$ is also continuous, and therefore $\varphi_f\in C_c(\grp)$. Due to the ring structure on each stalk, we can define a group homomorphism $\Phi:\secring\to C_c(\grp)$ such that $\Phi(f)=\varphi_f$, for $f$ a generator of $\secring$ as above.

On the other hand, given $u\in C_c(\grp)$ such that there exists $U$ compact-open bisection such that $u|_U$ is continuous and $\supp(u)\scj U$, we can define $a_u\in C_c(\grp^{(0)})$ by $a_u(x)=0$ if $x\notin\ran(U)$ and $a_u(x)=u((\ran|_U)^{-1}(x))$ if $x\in\ran(U)$. Then the function $\psi_u:\grp\to\OO$ given by $\psi_u(\gamma)=0$ if $\gamma\notin U$, and $\psi_u(\gamma)=[a_u,\ran(\gamma)]$ is such that $\psi_u|_U$ is continuous and therefore $\psi_u\in\secring$. Again, we can build a group homomorphism $\Psi:C_c(\grp)\to\secring$. Straightforward computations show that $\Psi=\Phi^{-1}$.

It remains to prove that the above group isomorphism is also a ring isomorphism. We prove that $\Psi$ is multiplicative. Since the maps are already group homomorphisms, due to distributivity, it is enough to consider $u,v\in C_c(G)$ such that there exist $U,V$ compact-open bisections such that $\supp(u)\scj U$, $\supp(v)\scj V$ and $u|_U$ and $v|_V$ are continuous. Notice that in this case $\supp(u*v)\scj UV$, $u*v|_{UV}$ is continuous and for $\gamma=\beta\rho$, where $\beta\in U$ and $\rho\in V$, we have that $(u*v)(\gamma)=u(\beta)v(\rho)$. Keeping the notation of the above paragraphs, we have that
\[\psi_{u*v}(\gamma)=\begin{cases}
[a_{u*v},\ran(\beta)], & \text{if }\gamma=\beta\rho\text{ with }\beta\in U,\ \gamma\in V\\
0, &\text{otherwise.}
\end{cases}\]
On the other hand,
\[(\psi_u *\psi_v)(\gamma)=\begin{cases}
[a_{u},\ran(\beta)]\alpha_{\beta}([a_{v},\ran(\rho)]), & \text{if }\gamma=\beta\rho\text{ with }\beta\in U,\ \gamma\in V\\
0, &\text{otherwise.}
\end{cases}
\]
Now, if $\gamma=\beta\rho$, with $\beta\in U$ and $\rho\in V$, then
\begin{align*}
    [a_{u},\ran(\beta)]\alpha_{\beta}([a_{v},\ran(\rho)]) &= [a_{u},\ran(\beta)][a_{v}\circ\dom\circ(\ran|_U)\inv ,\ran(\beta)] \\
    & = [a_{u}(a_{v}\circ\dom\circ(\ran|_U)\inv),\ran(\beta)].
\end{align*}
Notice that if $x\notin\ran(UV)$, then
\[a_{u*v}(x)=0=a_u(x)(a_{v}\circ\dom\circ(\ran|_U)\inv)(x).\]
And if $x=\ran(\gamma)$, where $\gamma=\beta\rho$, with $\beta\in U$ and $\rho\in V$, then
\begin{align*}
    a_{u*v}(x) &=(u*v)(\gamma) \\
    &=u(\beta)v(\rho)\\
    &=a_u(x)a_v(\ran(\rho))\\
    &=a_u(x)a_v(\dom(\beta))\\
    &=a_u(x)(a_{v}\circ\dom\circ(\ran|_U)\inv)(x).\\
\end{align*}
It follows that $\psi_{u*v}=\psi_u*\psi_v$ and hence the multiplicativity of $\Psi$ is proved.
\end{proof}


We finish by characterizing simplicity of $C_c(\grp)$

\begin{Lemma}
\label{intiso=intker}
Let $\grp$ be an ample groupoid and $\OO$ the $\grp$-sheaf constructed above. Then, $\mathrm{Int}(\ker(\OO))=\grp^{(0)}$ if and only if $\grp$ is effective.
\end{Lemma}

\begin{proof}
That $\grp$ effective implies $\mathrm{Int}(\ker(\OO))=\grp^{(0)}$ holds for any $\grp$-sheaf of rings. For the other implication, let $\gamma\in\mathrm{Int}(\mathrm{Iso}(\grp))$ be such that there exists a compact-open bisection $U$ with $\gamma\in U\scj\mathrm{Iso}(\grp)$. Observe that, in this case, $\dom\circ(\ran|_U)^{-1}$ is the identity map on $\ran(U)$.
Let $a\in C_c(\grp)$ and $\rho\in U$. Then,
\[\alpha_\rho([a,\dom(\rho)])=[a\circ\dom\circ(\ran|_U)\inv,\ran(\rho)]=[a,\dom(\rho)],\]
so that $\rho\in\ker(\OO)$. Hence $U\scj\ker(\OO)$ and $\gamma\in\mathrm{Int}(\ker(\OO))=\grp^{(0)}$. It follows that $\grp$ is effective.
\end{proof}

\begin{Thm}\label{thm:C_c(G) simple}
Let $\grp$ be an ample Hausdorff groupoid. Then, $C_c(\grp)$ is simple if and only if $\grp$ is minimal and effective.
\end{Thm}

\begin{proof}
Note that the equivalence of (1) and (2) in Theorem \ref{simplelife} holds for any $\grp$-sheaf because it comes from \cite[Theorem 3.7]{BGOR}. Let $\OO$ be the sheaf constructed above and consider the isomorphism of Theorem \ref{thm:groupoid.ring.sheaf}. Observe that $\OO$ is a sheaf of commutative rings that satisfies the hypothesis that every non zero section is invertible at some point. By Corollary \ref{c:minimal}, we have that $C_c(\grp^{(0)})$ is $\grp^a$-simple if and only if $\grp$ is minimal. If $\grp$ is effective, then $C_c(\grp^{(0)})$ is maximal commutative by Corollary \ref{c:effective.case.max.cen}. On the other hand if $C_c(\grp^{(0)})$ is maximal commutative, then $\grp$ is effective by Proposition \ref{p:neweffective} and Lemma \ref{intiso=intker}. Hence, we get that $C_c(\grp)$ is simple if and only if $\grp$ is minimal and effective.
\end{proof}

\begin{Example}[Graph algebras]
Let $E$ be a graph. By means of the graph groupoid $\grp_E$, it was shown in \cite{CFST} that the Leavitt path algebra $L_{\cn}(E)$ is dense in the C*-algebra $C^*(E)$. Observe that the algebra $C_c(\grp_E)$ is between $L_{\cn}(E)$ and $C^*(E)$. Since $\grp_E$ is Hausdorff, using Theorem \ref{thm:C_c(G) simple} and the results of \cite{BCFS}, we obtain that the following are equivalent
\begin{enumerate}
    \item $L_{\cn}(E)$ is simple;
    \item $C^*(E)$ is simple;
    \item $C_c(\grp_E)$ is simple.
\end{enumerate}
\end{Example}




\subsection{The algebra of continuous functions, with compact support, on the transformation groupoid.}

In this subsection, $X$ is always a Hausdorff, locally compact, totally disconnected topological space, and $C_c(X)$ denotes the algebra of all continuous, compactly supported, $R$-valued functions on $X$ (where $R$ stands for either the real numbers or the complex numbers), with point-wise addition and multiplication. 

It is proved in \cite{BG} that a partial skew group ring of the form $\mathcal L_C(X) \rtimes G$ (where $\mathcal L_C(X)$ stands for the locally constant functions on $X$) can be seen as the Steinberg algebra associated with the transformation groupoid $G \ltimes X$. Next, we argue that an analogous result holds when we replace $\mathcal L_C(X)$ with $C_c(X)$ and the Steinberg algebra with $C_c(G \ltimes X)$. In fact, the outline of the proof is the same as the one of the proof given in \cite[Theorem~3.2]{BG}, and so we will refrain from presenting a whole proof, and instead will only point to the main differences between the two settings.

For $f \in C_c(X)$ we define the support of $f$ by $$\supp(f)= \overline{\{x \in X \ :  \ f(x)\neq 0\}}.$$ Notice that when dealing with a function $f$ in $\mathcal L_C(X)$, the set $\{x \in X \ :  \ f(x)\neq 0\}$ is already closed and so it is not necessary to take closure in the definition of support. This is one of the main differences between the $C_c(X)$ and $\mathcal L_C(X)$ cases. 

Now, let $\theta=(\{X_g\}_{g \in G}, \{\theta_g\}_{g \in G})$ be a partial action of a discrete group $G$ on $X$, such that $X_g$ is clopen for every $g$ in $G.$ Such action induces an action in the algebra level, as done in \cite{Beuter} and \cite{Doke}: For each $g$ in $G,$ consider the ideal $D_g:= \{ f \in C_c(X) \, :  \, f \ \mbox{ vanishes on } X\setminus X_g \}$ in $C_c(X)$, and define $\alpha_g: D_{g^{-1}} \rightarrow D_g$ by setting $\alpha_g(f)=f\circ \theta_{g^{-1}},$ for all $ f \in D_{g^{-1}}.$ Then the collection 
\begin{equation}
\alpha := ( \{D_g\}_{g \in G}, \{\alpha_g\}_{g \in G} )
\end{equation}
is an algebraic partial action of $G$ on $C_c(X)$.

Associated with the above partial action we consider the partial skew group ring $C_c(X)\rtimes G$. We also associate to the action $\theta$ an \'{e}tale groupoid, denoted by $G \ltimes X$, and known as the \emph{transformation grupoid} (see \cite{Abadie}). Let
 $$G \ltimes X:= \{(t,x) \, : \, t \in G \,\, \mbox{and} \,\, x \in X_t \}.$$
The inverse of $(t, x) \in G \ltimes X$ is 
$$(t, x)^{-1}=(t^{-1}, \theta_{t^{-1}}(x)).$$ So, the range and source maps $r: G \ltimes X \rightarrow  (G \ltimes X)^{(0)}$ and $s : G \ltimes X \rightarrow (G \ltimes X)^{(0)}$ are given by $r(t,x) =(1,x)$, and $s(t,x)=(1, \theta_{t^{-1}}(x))$ (where $1$ denotes the group unit). 

From the above, we have that $(s,y), (t,x) \in G \ltimes X$ then $(s,y), (t,x) $ is a composable pair if, and only if, $\theta_{s^{-1}}(y)=x.$ In this case, we have
$$(s,y)(t,x)=(st, y).$$
Finally, we equip $G \ltimes X$ with the topology inherited from the product topology on $G \times X.$ 

We now state the key result in this subsection.

\begin{Thm}\label{theorisomortransfgrou}
Let $\theta=(\{X_g\}_{g \in G}, \{\theta_g\}_{g \in G})$ be a partial action of a discrete group $G$ over a locally compact, Hausdorff, totally disconnected topological space $X$, such that each $X_g$ is clopen for all $g$. Let $(\{D_g\}_{g \in G}, \{\alpha_g\}_{g\in G})$ be the corresponding partial action (as defined above) and $G \ltimes X$ be the transformation groupoid associate with $\theta.$ Then, $C_c(X)\rtimes G$ and $C_c(G \ltimes X)$ are isomorphic as $R$-algebras. 
\end{Thm}
\begin{proof}
The proof is analogous to the proof of \cite[Theorem~3.2]{BG}. One should only take into account that for a general continuous function the set of points where the function does not vanish does not need be closed. 
\end{proof}

\bibliographystyle{abbrv}
\bibliography{ref}

\begin{thebibliography}{10}

\bibitem{Abadie}
F.~Abadie.
\newblock On partial actions and groupoids.
\newblock {\em Proc. Amer. Math. Soc.}, 132(4):1037--1047, 2004.

\bibitem{BCS}
A.~Baraviera, W.~Cortes, and M.~Soares.
\newblock Simplicity of crossed products by twisted partial actions.
\newblock {\em J. Aust. Math. Soc.}, 108(2):202--225, 2020.

\bibitem{Beuter}
V.~Beuter and D.~Gon\c{c}alves.
\newblock Partial crossed products as equivalence relation algebras.
\newblock {\em Rocky Mountain J. Math.}, 46:85--104, 2016.

\bibitem{BGOR}
V.~Beuter, D.~Gon\c{c}alves, J.~\"{O}inert, and D.~Royer.
\newblock Simplicity of skew inverse semigroup rings with applications to
  {S}teinberg algebras and topological dynamics.
\newblock {\em Forum Math.}, 31(3):543--562, 2019.

\bibitem{BG}
V.~M. Beuter and D.~Gon\c{c}alves.
\newblock The interplay between {S}teinberg algebras and skew rings.
\newblock {\em J. Algebra}, 497:337--362, 2018.

\bibitem{BCFS}
J.~Brown, L.~O. Clark, C.~Farthing, and A.~Sims.
\newblock Simplicity of algebras associated to \'{e}tale groupoids.
\newblock {\em Semigroup Forum}, 88(2):433--452, 2014.

\bibitem{CEP}
L.~O. Clark, R.~Exel, and E.~Pardo.
\newblock A generalized uniqueness theorem and the graded ideal structure of
  {S}teinberg algebras.
\newblock {\em Forum Math.}, 30(3):533--552, 2018.

\bibitem{CFST}
L.~O. Clark, C.~Farthing, A.~Sims, and M.~Tomforde.
\newblock A groupoid generalisation of {L}eavitt path algebras.
\newblock {\em Semigroup Forum}, 89(3):501--517, 2014.

\bibitem{ConnesNH}
A.~Connes.
\newblock A survey of foliations and operator algebras.
\newblock Operator algebras and applications, {Proc}. {Symp}. {Pure} {Math}.
  38, {Part} 1, {Kingston}/{Ont}. 1980, 521-628 (1982)., 1982.

\bibitem{BC}
L.~G. Cordeiro and V.~Beuter.
\newblock The dynamics of partial inverse semigroup actions.
\newblock {\em J. Pure Appl. Algebra}, 224(3):917--957, 2020.

\bibitem{Paulinho}
P.~Demeneghi.
\newblock The ideal structure of steinberg algebras.
\newblock {\em Adv. Math.}, 3525:777--835, 2019.

\bibitem{Doke}
M.~Dokuchaev and R.~Exel.
\newblock The ideal structure of algebraic partial crossed products.
\newblock {\em Proc. Lond. Math. Soc.}, 115(1):91--134, 2017.

\bibitem{Dowker}
C.~H. Dowker.
\newblock {\em Lectures on sheaf theory}.
\newblock Tata Institute of Fundamental Research, Bombay, 1956.
\newblock Notes by S. V. Adavi and N. Ramabhadran.

\bibitem{Effros}
E.~G. Effros and F.~Hahn.
\newblock Locally compact transformation groups and c*- algebras.
\newblock {\em Mem. Amer. Math. Soc.}, 75, 1967.

\bibitem{ExelPardo}
R.~Exel and E.~Pardo.
\newblock The tight groupoid of an inverse semigroup.
\newblock {\em Semigroup Forum}, 92(1):274--303, 2016.

\bibitem{BenDan}
D.~Gon{\c{c}}alves and B.~Steinberg.
\newblock {\'E}tale groupoid algebras with coefficients in a sheaf and skew
  inverse semigroup rings.
\newblock {\em Canad. J. Math.}, to appear.

\bibitem{DG}
D.~Gon\c{c}alves.
\newblock Simplicity of partial skew group rings of abelian groups.
\newblock {\em Canad. Math. Bull.}, 57(3):511--519, 2014.

\bibitem{GOR}
D.~Gon\c{c}alves, J.~\"{O}inert, and D.~Royer.
\newblock Simplicity of partial skew group rings with applications to {L}eavitt
  path algebras and topological dynamics.
\newblock {\em J. Algebra}, 420:201--216, 2014.

\bibitem{Hamsher}
R.~M. Hamsher.
\newblock Commutative rings over which every module has a maximal submodule.
\newblock {\em Proc. Amer. Math. Soc.}, 18:1133--1137, 1967.

\bibitem{johnstone}
P.~T. Johnstone.
\newblock {\em Stone spaces}, volume~3 of {\em Cambridge Studies in Advanced
  Mathematics}.
\newblock Cambridge University Press, Cambridge, 1986.
\newblock Reprint of the 1982 edition.

\bibitem{cp}
Y.~P. L.~O.~Clark.
\newblock Kumjian--pask algebras of finitely-alighed higher-rank graphs.
\newblock {\em J. Algebra}, 482:364--397, 2017.

\bibitem{Nystedt1}
P.~Nystedt.
\newblock Simplicity of algebras via epsilon-strong systems.
\newblock {\em Colloq. Math.}, 162(2):279--301, 2020.

\bibitem{NYSTEDT2018433}
P.~Nystedt, J.~Öinert, and H.~Pinedo.
\newblock Artinian and noetherian partial skew groupoid rings.
\newblock {\em Journal of Algebra}, 503:433--452, 2018.

\bibitem{Park}
J.~K. Park.
\newblock Artinian skew group rings.
\newblock {\em Proc. Amer. Math. Soc.}, 75(I):279--301, 1979.

\bibitem{PatersonBook}
A.~L.~T. Paterson.
\newblock {\em Groupoids, inverse semigroups, and their operator algebras},
  volume 170 of {\em Progress in Mathematics}.
\newblock Birkh\"auser Boston, Inc., Boston, MA, 1999.

\bibitem{Pierce}
R.~S. {Pierce}.
\newblock {\em {Modules over commutative regular rings}}, volume~70.
\newblock Providence, RI: American Mathematical Society (AMS), 1967.

\bibitem{HRove}
S.~W. Rigby and T.~van~den Hove.
\newblock A classification of ideals in {S}teinberg and {L}eavitt path algebras
  over arbitrary rings.
\newblock {\em arXiv:2103.02712 [math.RA]}, 2021.

\bibitem{Steinbergalgebra}
B.~Steinberg.
\newblock A groupoid approach to discrete inverse semigroup algebras.
\newblock {\em Adv. Math.}, 223(2):689--727, 2010.

\bibitem{groupoidbundles}
B.~Steinberg.
\newblock Modules over \'etale groupoid algebras as sheaves.
\newblock {\em J. Aust. Math. Soc.}, 97(3):418--429, 2014.

\bibitem{groupoidprimitive}
B.~Steinberg.
\newblock Simplicity, primitivity and semiprimitivity of \'{e}tale groupoid
  algebras with applications to inverse semigroup algebras.
\newblock {\em J. Pure Appl. Algebra}, 220(3):1035--1054, 2016.

\bibitem{groupoidprime}
B.~Steinberg.
\newblock Prime \'{e}tale groupoid algebras with applications to inverse
  semigroup and {L}eavitt path algebras.
\newblock {\em J. Pure Appl. Algebra}, 223(6):2474--2488, 2019.

\bibitem{Ben}
B.~Steinberg.
\newblock Ideals of {\'e}tale groupoid algebras and {E}xel's {E}ffros-{H}ahn
  conjecture.
\newblock {\em J. Noncommut. Geom.}, to appear.

\end{thebibliography}

\end{document}